\documentclass[12pt,reqno]{amsart}

\usepackage{amslfmt}
\usepackage{packages}
\usepackage{macros}

\begin{document}

\title[Rarefactions, Jumps, and Constants]{%
  When Do Riemann Solutions Consist of \\
  Rarefactions, Jumps, and Constants?
}

\author[Plohr]{Bradley J. Plohr}
\address{%
  Los Alamos, NM 87544
}
\email{bradley.j.plohr@gmail.com}

\author[Schecter]{Stephen Schecter}
\address{%
  Mathematics Department \\
  North Carolina State University \\
  Raleigh, NC 27695
}
\email{schecter@ncsu.edu}

\author[Marchesin]{Dan Marchesin}
\address{%
  Instituto de Matem\'atica Pura e Aplicada \\
  Estrada Dona Castorina 110 \\
  22460 Rio de Janeiro, RJ, Brazil
}
\email{marchesi@impa.br}

\date{\today}




\begin{abstract}
  A solution of a Riemann problem
  for a strictly hyperbolic system of conservation laws
  is traditionally expected to consist of
  rarefaction waves, jump discontinuities, and constant states.
  In this paper,
  we investigate whether a Riemann solution has this structure
  when the solution is only assumed to be measurable and essentially bounded.
  To discriminate continuous and discontinuous features
  in an $L^\infty$ solution,
  we introduce one-sided accumulation sets based on local essential images.
  Supposing that throughout a bounded open interval
  a solution is continuous
  in the essential image (\ESSIM{}) sense,
  we prove that it is a rarefaction wave
  if it is resonant (the characteristic speed equals $x/t$),
  and otherwise it is constant.
  Although an \ESSIM{} discontinuity
  might not be a jump discontinuity,
  we show that all \ESSIM{} accumulation states
  lie on a common Hugoniot locus and have the same speed.
  Anomalies are possible if there are limit points of \ESSIM{} discontinuities,
  but if the set of \ESSIM{} discontinuities is finite,
  then an $L^\infty$ Riemann solution has bounded variation
  and is composed of finitely many
  rarefaction waves, jump discontinuities, and constant states.
\end{abstract}

\maketitle

\section{Introduction}
\label{sec:introduction}

Ever since Riemann's seminal work~\cite{Rie60} on gas dynamics,
self-similar solutions of systems of conservation laws,
which we call Riemann solutions,
have been constructed by assembling rarefaction waves,
jump discontinuities, and constant states.
These features correspond to
mathematical properties of the solution profile.
In a rarefaction wave,
the profile varies continuously,
fanning out along characteristic lines.
A jump discontinuity in the profile propagates at a speed constrained by
the Rankine--Hugoniot condition.
In a constant state, the solution is steady and homogeneous.
If a Riemann solution is piecewise smooth,
then only these solution features occur.
In this paper,
we examine whether a Riemann solution has this structure
under the milder hypothesis that it is measurable and essentially bounded.

Dafermos~\cite{Daf08} considered this question
for strictly hyperbolic and genuinely nonlinear systems.
He assumed that the solution $v$,
as a function of the speed variable $\xi := x/t$,
has bounded variation (BV) and small oscillation.
By leveraging an integral equation satisfied by a solution,
he showed that $v$ is the sum of a jump function and a Lipschitz function.
He showed that the domain of $v$ consists of three disjoint sets:
the complement $\mC$ of the support of $d v/d\xi$, where $v$ is
locally constant;
the set $\mS$ of jump discontinuities of $v$, which satisfy the
Rankine--Hugoniot condition;
and the remaining set $\mW$, in which $v$ varies continuously,
$\xi$ equals a characteristic speed for $v(\xi)$,
and $d v/d\xi$ is a corresponding eigenvector almost everywhere (a.e.).

In his monograph~\cite[\S~9.1]{Daf26},
Dafermos applied much the same strategy to solutions
that are measurable and essentially bounded
instead of BV.
However, we find that extension to such $L^\infty$ solutions requires
different definitions for the sets $\mS$, $\mW$, and $\mC$,
as we explain in
Sec.~\ref{sec:continuity_for_weak_solutions}.

The present paper concerns strictly hyperbolic systems
that are not required to be genuinely nonlinear.
We treat a scale-invariant weak solution $v \in L^\infty$
without assuming that it is a small oscillation solution.
Our definitions of continuity and discontinuity
for members of $L^\infty$
are based on the concept of essential image (abbreviated \ESSIM{}).
Crucially, the following properties hold:
if $v$ is \ESSIM{} continuous at each point of an open interval,
then $v$ is equal a.e.\ to a function $\overline{v}$
that is continuous within this interval in the usual sense;
and at an \ESSIM{} discontinuity point,
$v$ has distinct left and right \ESSIM{} accumulation sets.

Associated to $v$ are \ESSIM{} variants of the sets $\mC$, $\mS$, and $\mW$.
Following Dafermos,
we show that the integral equation satisfied by $v$
implies that $v$ is constant within each connected component of $\mC$.
We extend this result:
within each bounded open interval contained in $\mW$,
$v$ is a BV rarefaction wave
(as generalized to accommodate loss of genuine nonlinearity).
Central to the proof is
locally straightening the rarefaction integral curves
of one of the characteristic families.

We also show that an \ESSIM{} discontinuity point $\xi \in \mS$
has at least two distinct accumulation states;
this result entails that $\mS \sqcup \mW$ is closed and $\mC$ is open.
All accumulation states for $v$ at $\xi \in \mS$
are restricted to lie on a common Hugoniot locus and have the same speed.
Moreover, if $\xi \in \mS$ is isolated,
then $v$ has unique left and right accumulation states at $\xi$,
\ie $v$ has an \ESSIM{} jump discontinuity.

Nonetheless, $v$ can have anomalous features,
such as limit points of $\mS$
or infinitely many connected components of $\mC$ or $\mW$.
For example~\cite{PloSchMar26b},
an infinite sequence of under-compressive shock waves
can accumulate at an \ESSIM{} discontinuity that is not a jump discontinuity.
However, if $\mS$ is finite,
a Riemann solution comprises finitely many
rarefaction waves, jump discontinuities, and constant states.

The system of conservation laws~\eqref{eq:conslaw}
often models a physical process
that exhibits nearly discontinuous features,
such as shock waves, contact discontinuities,
and under-compressive waves,
which are subject to physical principles beyond conservation.
To serve as idealizations of these features,
jump discontinuities must be constrained by admissibility conditions
that reflect these principles.
However, as our purpose is only to explore
the mathematical possibilities for self-similar solutions,
we do not impose any admissibility conditions on discontinuities.

For analysis involving $L^\infty$ functions,
we find it critical to employ \ESSIM{} accumulation sets,
which are local geometric objects that package measure-theoretic information
and respect $L^\infty$ equivalence classes.
While the essential image of a measurable function is a classical concept,
\ESSIM{} accumulation sets are rarely defined;
they have received systematic development
only for real-valued functions~\cite{FelWag98}.
In Sec.~\ref{sec:continuity_for_weak_solutions}
and Appendix~\ref{sec:essential_image},
we explain the \ESSIM{} approach
and compare it to two alternatives
that prove inadequate for $L^\infty$ conservation law theory.

In outline, the paper is as follows.
Section~\ref{section:systems_of_conservation_laws}
recalls definitions from the theory of conservation laws.
In Sec.~\ref{section:integral_equation},
we derive an integral equation satisfied by a self-similar solution;
it facilitates the proof of solution regularity and structure.
In Sec.~\ref{section:continuity}
and Appendix~\ref{sec:essential_image}
we define \ESSIM{} accumulation sets and continuity.
As shown in Sec.~\ref{section:constant_states},
the integral equation implies that a solution is constant
in any interval in which it is \ESSIM{} continuous and non-resonant.
This argument is extended to characterize rarefaction waves
in Sec.~\ref{section:rarefaction_waves},
the principal idea being to straighten the rarefaction integral curves
of the relevant family locally.
The integral equation also constrains
the accumulation states of an \ESSIM{} discontinuity point
to lie on a common Hugoniot locus and have the same speed.
Properties of discontinuities are derived in
Sec.~\ref{section:discontinuities}.
In Sec.~\ref{section:solution_structure},
we determine the structure of a self-similar solution
by partitioning the speed axis into four disjoint sets
corresponding to isolated \ESSIM{} jump discontinuities,
BV rarefaction waves, constant states,
and limit points of \ESSIM{} discontinuities.
Finally, in Sec.~\ref{section:conclusion},
we show that,
if a solution has no such limit points,
then the jump discontinuities, rarefaction waves, and constant states
are finite in number.

\section{Systems of conservation laws}
\label{section:systems_of_conservation_laws}

\subsection{Systems of conservation laws}
\label{sec:systems_of_conservation_laws}

A \emph{system of conservation laws} in one spatial dimension
is a system of partial differential equations of the form
\begin{equation}
  u_t + f(u)_x = 0.
  \label{eq:conslaw}
\end{equation}
Here $x \in \Rset$ is spatial position and $t > 0$ is time;
subscripts indicate partial derivatives.
The \emph{flux function} $f : \bbU \to \Rset^n$
is defined on an open subset
$\bbU \subseteq \Rset^n$, called \emph{state space};
elements of $\bbU$ are called \emph{states}.
The flux function $f$ is assumed to be $C^2$;
its derivative, the \emph{characteristic matrix},
is denoted
\begin{equation}
  A := D f.
  \label{eq:characteristic_matrix}
\end{equation}
The \emph{strictly hyperbolic region} $\bbU_{SH}\subseteq\bbU$
comprises states at which the characteristic matrix has real,
distinct eigenvalues.
We assume that that the convex hull of $\bbU_{SH}$ is contained in $\bbU$.
(This hypothesis is invoked in
Sec.~\ref{sec:integral_equation}.)

We consider a measurable and essentially bounded function $u$ that is
defined on the open upper half-plane in space--time,
$\bbH := \Rset \times (0, \infty)$,
takes values in the strictly hyperbolic region,
and solves system~\eqref{eq:conslaw}
in the weak sense.
More precisely,
we require that:
\begin{enumerate}
  \item[(a)] $u \in L^\infty(\bbH; \Rset^n)$;
  \item[(b)] the essential image of $u$
    (see Def.~\ref{def:essential_image})
    is contained in $\bbU_{SH}$; and
  \item[(c)] $u$ is a weak solution of system~\eqref{eq:conslaw}.
\end{enumerate}

For convenience, we adopt the following terminology:
$u$ is a \emph{state function} if it has properties~(a) and~(b).

\begin{lemma}
  \label{lem:well-defined}
  Suppose that $u$ is a state function.
  Then $f \comp u \in L^\infty(\bbH; \Rset^n)$.
\end{lemma}

\begin{proof}
  By property~(b) and Lemma~\ref{lem:essential_image},
  $u(x, t) \in \bbU_{SH}$ for almost every $(x, t) \in \bbH$.
  The domain of $f$ contains $\bbU_{SH}$,
  so $f\bigl(u(x, t)\bigr)$ is defined for almost every $(x, t) \in \bbH$.
  As $f$ is continuous,
  $f \comp u$ is measurable.
  By property~(a) and Lemma~\ref{lem:essential_image},
  the essential image of $u$ is compact.
  Because $f$ is bounded on this compact set,
  $f \comp u$ is essentially bounded.
\end{proof}

Let $u$ be a state function.
Then for any \emph{test function} $\phi \in C^1_c(\bbH; \Rset^n)$,
$\phi_t \cdot u$ and $\phi_x \cdot (f\comp u)$ are integrable.
(The dot denotes the Euclidean inner product in $\Rset^n$.)
Therefore, the following definition makes sense.

\begin{definition}
  \label{def:weak_solution}
  A \emph{weak solution}
  of system~\eqref{eq:conslaw}
  is a state function $u$ such that
  \begin{equation}
    \int_0^\infty \int_{-\infty}^\infty
    \bigl[\phi_t\cdot u + \phi_x\cdot f(u)\bigr] \,d x\,d t = 0
    \ \text{for all $\phi \in C^1_c(\bbH; \Rset^n)$.}
    \label{eq:weak_condition}
  \end{equation}
\end{definition}

\begin{remark}
  \label{rem:no_initial_condition}
  We are not concerned with initial conditions
  for system~\eqref{eq:conslaw},
  so we confine the support of the test function $\phi$
  to the space-time region $t > 0$.
  \qed
\end{remark}

\subsection{Self-similar solutions}
\label{sec:self-similar_solutions}

Our focus is on solutions that are scale-invariant.
A function $w : \bbH \to \bbU$ is scale-invariant
if $w(\alpha\,x, \alpha\,t) = w(x, t)$ for all $\alpha > 0$.
In particular, taking $t = 1/\alpha$ shows that $w(x, t) = w(x/t, 1)$.

From a different perspective,
a scale-invariant function $w$ is self-similar:
for all $\alpha > 0$, its profile $w(x, \alpha\,t)$ at time $\alpha\,t$
coincides with the spatially rescaled profile $w(x/\alpha, t)$ at time $t$.

Adapting to bounded measurable functions,
we say that $u \in L^\infty(\bbH; \Rset^n)$ is \emph{self-similar}
provided there exists $v \in L^\infty(\Rset; \Rset^n)$ such that
\begin{equation}
  \text{$u(x, t) = v(x/t)$ for almost every $(x, t) \in \bbH$}.
  \label{eq:u_in_terms_of_v}
\end{equation}
Equivalently,
\begin{equation}
  \text{$u(\xi\,\tau, \tau) = v(\xi)$
  for almost every $(\xi, \tau)\in \Rset \times (0, \infty)$.}
  \label{eq:u_in_terms_of_v_alt}
\end{equation}
In turn, $v$ can be recovered from $u$ as follows.
Let $\chi \in C^1_c\bigl((0, \infty)\bigr)$ have integral equal to~1.
Then $(\xi, \tau) \mapsto u(\xi\,\tau, \tau)\,\chi(\tau)$ is locally integrable
over $\Rset \times (0, \infty)$.
By Fubini's theorem~\cite[Theorem~1.22]{EvaGar15},
the function
\begin{equation}
  \xi \mapsto \int_0^\infty u(\xi\,\tau, \tau)\,\chi(\tau)\,d\tau
  \label{eq:v_in_terms_of_u}
\end{equation}
is measurable on $\Rset$,
and by Eq.~\eqref{eq:u_in_terms_of_v_alt},
it equals $v(\xi)$ for almost every $\xi \in \Rset$.

We define $v$ to be a \emph{reduced state function}
provided $v \in L^\infty(\Rset; \Rset^n)$ and $\essim v \subseteq \bbU_{SH}$.

\begin{lemma}
  \label{lem:self-similar_integral_equality}
  Suppose $u$ is a self-similar state function
  and $v$ is the reduced state function related to $u$ by
  Eq.~\eqref{eq:u_in_terms_of_v}.
  Let $\phi \in C^1_c(\bbH; \Rset^n)$, and for $\xi \in \Rset$ set
  \begin{equation}
    \psi(\xi) := \int_0^\infty \phi(\xi\,\tau, \tau)\,d\tau.
    \label{eq:test_function_map}
  \end{equation}
  Then $\psi \in C^1_c(\Rset; \Rset^n)$ and
  \begin{equation}
    \int_0^\infty \int_{-\infty}^\infty
    \bigl[\phi_t\cdot u + \phi_x\cdot f(u)\bigr] \,d x\,d t
    = \int_{-\infty}^\infty \left\{\frac{d\psi}{d\xi}\cdot
    \left[-\xi\,v + f(v)\right] - \psi\cdot v\right\} d\xi.
    \label{eq:equality_self-similar_integrals}
  \end{equation}
\end{lemma}

\begin{proof}
  In the integral on the left-hand side of
  Eq.~\eqref{eq:equality_self-similar_integrals},
  change variables of integration from $(x, t) \in \bbH$
  to $(\xi, \tau) := (x/t, t) \in \Rset \times (0, \infty)$,
  replacing $(x, t)$ by $(\xi\,\tau, \tau)$,
  $u(\xi\,\tau, \tau)$ by $v(\xi)$,
  and $d x\,d t$ by $\tau\,d\xi\,d\tau$
  (the Jacobian determinant is $\tau$).
  The result is
  \begin{equation}
    \begin{split}
      & \int_0^\infty \int_{-\infty}^\infty
      \left[\phi_t(\xi\,\tau, \tau) \cdot v(\xi)
        + \phi_x(\xi\,\tau, \tau)
      \cdot f\bigl(v(\xi)\bigr)\right] \tau\,d\xi\,d\tau \\
      & \qquad = \int_{-\infty}^\infty
      \left[\int_0^\infty \tau\,\phi_t(\xi\,\tau, \tau)\,d\tau \cdot v(\xi)
        + \int_0^\infty \tau\,\phi_x(\xi\,\tau, \tau)\,d\tau
      \cdot f\bigl(v(\xi)\bigr)\right] d\xi,
    \end{split}
    \label{eq:weak_integral_transformed}
  \end{equation}
  where we have switched order of integration.
  The two integrals with respect to $\tau$ are functions of $\xi$
  that we relate to the function $\psi : \Rset \to \Rset^n$ defined
  by Eq.~\eqref{eq:test_function_map}.

  First, $\psi \in C^1_c(\Rset; \Rset^n)$.
  Indeed,
  the Leibniz integral rule
  implies that $\psi$ is differentiable and its derivative
  \begin{equation}
    \frac{d\psi}{d\xi} = \int_0^\infty \tau\,\phi_x(\xi\,\tau, \tau)\,d\tau
    \label{eq:derivative_of_psi}
  \end{equation}
  is continuous.
  Moreover, $\psi$ has compact support:
  because $\supp \phi$ is a compact subset of $\bbH$,
  the set $\{\, x/t \,:\, (x, t) \in \supp \phi \,\}$
  is contained in some compact interval $[\xi_L, \xi_R]$;
  and if $\xi < \xi_L$ or $\xi > \xi_R$,
  then $\phi(\xi\,\tau, \tau) = 0$ for all $\tau > 0$,
  so that $\psi(\xi) = 0$.

  By Eq.~\eqref{eq:derivative_of_psi},
  the second integral with respect to $\tau$ inside
  the integral~\eqref{eq:weak_integral_transformed}
  is $d\psi/d\xi$.
  Also,
  \begin{equation}
    \begin{split}
      \int_0^\infty \tau\,\phi_t(\xi\,\tau, \tau)\,d\tau
      & = \int_0^\infty \tau\,\frac{d}{d\tau}\phi(\xi\,\tau, \tau)\,d\tau
      - \xi \int_0^\infty \tau\,\phi_x(\xi\,\tau, \tau)\,d\tau             \\
      & = \tau\,\phi(\xi\,\tau, \tau) \Bigl|_{\tau = 0}^\infty
      -\psi - \xi\,\frac{d\psi}{d\xi}
      = -\psi - \xi\,\frac{d\psi}{d\xi},
    \end{split}
  \end{equation}
  where we have integrated by parts
  and used Eq.~\eqref{eq:derivative_of_psi}
  again.
  Thus, integral~\eqref{eq:weak_integral_transformed}
  coincides with the integral on the right-hand side of
  Eq.~\eqref{eq:equality_self-similar_integrals}.
\end{proof}

Consequently, for a self-similar state function,
condition~\eqref{eq:weak_condition}
amounts to the following condition on its corresponding
reduced state function~\cite[\S~9.1]{Daf26}.

\begin{proposition}
  \label{prop:Self-Similar_Weak_Condition}
  Suppose $u$ is a self-similar state function
  and $v$ is the reduced state function related to $u$ by
  Eq.~\eqref{eq:u_in_terms_of_v}.
  Then $u$ is a weak solution of system~\eqref{eq:conslaw}
  if and only if
  \begin{equation}
    \int_{-\infty}^\infty
    \left\{\frac{d\psi}{d\xi}\cdot
    \left[-\xi\,v + f(v)\right] - \psi\cdot v\right\} d\xi = 0
    \ \text{for all $\psi \in C^1_c(\Rset; \Rset^n)$.}
    \label{eq:self-similar_weak_condition}
  \end{equation}
\end{proposition}

\begin{proof}
  Suppose that $v$ satisfies
  condition~\eqref{eq:self-similar_weak_condition}.
  To demonstrate that $u$ is a weak solution,
  consider any $\phi \in C^1_c(\bbH; \Rset^n)$.
  With $\psi$ defined by Eq.~\eqref{eq:test_function_map},
  the integral in Eq.~\eqref{eq:self-similar_weak_condition}
  vanishes.
  Therefore,
  equality~\eqref{eq:equality_self-similar_integrals}
  implies that condition~\eqref{eq:weak_condition} holds for $\phi$.

  Conversely,
  suppose that $u$ is a weak solution and specify $\psi \in
  C^1_c(\Rset; \Rset^n)$.
  We take $\chi \in C^1_c\bigl((0, \infty)\bigr)$ to have integral equal to~1
  and set $\phi(x, t) := \chi(t)\,\psi(x/t)$;
  then $\phi \in C^1_c(\bbH; \Rset^n)$ and
  Eq.~\eqref{eq:test_function_map}
  holds.
  As the integral in condition~\eqref{eq:weak_condition}
  vanishes for this choice of $\phi$,
  the integral in condition~\eqref{eq:self-similar_weak_condition}
  vanishes for the specified $\psi$.
\end{proof}

\begin{definition}
  \label{def:self-similar_weak_solution}
  A reduced state function $v$
  is a \emph{self-similar weak solution} of system~\eqref{eq:conslaw}
  if condition~\eqref{eq:self-similar_weak_condition}
  holds.
\end{definition}

\subsection{Features in solutions}
\label{sec:features_in_solutions}

The purpose of this section is to motivate and orient our investigation
by describing the principal features in a self-similar weak solution $v$.
To simplify the discussion,
we make the strong regularity assumption that $v$ is piecewise $C^1$.
We identify three features,
namely, smooth rarefaction waves, isolated jump discontinuities, and
constant states,
which traditionally serve as building blocks for constructing Riemann solutions.
However,
in the rest of the paper,
we omit this regularity assumption
and investigate when (generalizations of) these features
occur in $L^\infty$ solutions.

\subsubsection*{Smooth solutions}
Suppose that $v$ is $C^1$ on an open interval $J$.
Integrating by parts in Eq.~\eqref{eq:self-similar_weak_condition}
and using the definition $A := D f$ shows that
\begin{equation}
  \int_{-\infty}^\infty
  \psi\cdot \left[-\xi\,I + A(v)\right] \frac{d v}{d\xi}\,d\xi = 0
  \label{eq:smooth_self-similar_condition}
\end{equation}
for all test functions $\psi \in C^1_c(\Rset; \Rset^n)$ with support in $J$.
As $v$ is $C^1$,
\begin{equation}
  \left[-\xi\,I + A(v)\right] \frac{d v}{d\xi} = 0
  \label{eq:eigenvalue_eigenvector}
\end{equation}
for all $\xi \in J$.
For each point $\xi_\star \in J$,
there are three cases.
\begin{enumerate}
  \item[(1)] $\det \left[-\xi_\star\,I +
    A\bigl(v(\xi_\star)\bigr)\right] \ne 0$:
    For all $\xi$ in an open subinterval of $J$ containing $\xi_\star$,
    the matrix $-\xi\,I + A\bigl(v(\xi)\bigr)$ is invertible,
    hence $d v / d\xi = 0$ throughout that subinterval.
    Consequently,
    $v$ is a \emph{constant state} in this subinterval.

  \item[(2)] $d v / d\xi \ne 0$ at $\xi_\star$:
    For all $\xi$ in a subinterval of $J$ containing $\xi_\star$,
    $d v / d\xi \ne 0$ at $\xi$.
    By Eq.~\eqref{eq:eigenvalue_eigenvector},
    $d v/d\xi$ at $\xi$ is an eigenvector of $A\bigl(v(\xi)\bigr)$
    and $\xi$ is the corresponding eigenvalue.
    As explained in more detail in
    Sec.~\ref{sec:characteristic_speed},
    there exist a neighborhood $\Uopen$ of $u_\star := v(\xi_\star)$
    and $C^1$ maps $\lambda : \Uopen \to \Rset$ and $r : \Uopen
    \to \Rset^n$
    such that $\lambda(u_\star) = \xi_\star$
    and $r(u)$ is an eigenvector of $A(u)$ with eigenvalue $\lambda(u)$
    for all $u \in \Uopen$.
    Necessarily, the \emph{resonance condition}
    \begin{equation}
      \xi = \lambda\bigl(v(\xi)\bigr)
      \label{eq:resonance}
    \end{equation}
    holds for all $\xi$ in a possibly smaller subinterval containing
    $\xi_\star$,
    and $\xi \mapsto v(\xi)$ is a reparametrization of an integral curve of $r$.
    Within this smaller subinterval,
    the solution $v$ is called a \emph{smooth rarefaction wave}.

  \item[(3)] $\det \left[-\xi_\star\,I + A\bigl(v(\xi_\star)\bigr)\right] = 0$
    and $d v / d\xi = 0$ at $\xi_\star$:
    Such a point lies between a constant state
    and either a smooth rarefaction wave
    or another constant state
    (see Prop.~\ref{prop:E}).
\end{enumerate}

\subsubsection*{Discontinuous solutions}
Suppose that $v$ has an isolated jump discontinuity at $\xi_\star$,
in that there exist $C^1$ solutions $v_\ell$ on $(-\infty, \xi_\star)$
and $v_r$ on $(\xi_\star, \infty)$ such that
\begin{equation}
  v(\xi) =
  \begin{cases}
    v_\ell(\xi) & \text{if $\xi < \xi_\star$,} \\
    v_r(\xi)    & \text{if $\xi > \xi_\star$}
  \end{cases}
  \label{eq:single_jump}
\end{equation}
and the limits
\begin{align}
  \text{$u^- := \lim_{\zeta \incr \xi_\star} v_\ell(\zeta)$
  \quad and \quad $u^+ := \lim_{\zeta \decr \xi_\star} v_r(\zeta)$}
  \label{eq:one-sided_limits}
\end{align}
exist with $u^+ \ne u^-$.
Breaking the integral in condition~\eqref{eq:self-similar_weak_condition}
into integrals over $(-\infty, \xi_\star)$ and $(\xi_\star, \infty)$
and integrating by parts reduces it to
\begin{equation}
  -\psi(\xi_\star)\cdot\left\{-\xi_\star\,u^+ + f(u^+)
  - \left[-\xi_\star\,u^- + f(u^-)\right]\right\} = 0.
\end{equation}
Because the value $\psi(\xi_\star)$ is arbitrary,
a function of the form~\eqref{eq:single_jump}
is a weak solution if and only if the \emph{Rankine--Hugoniot condition}
\begin{equation}
  -s\,(u^+ - u^-) + f(u^+) - f(u^-) = 0
  \label{eq:Rankine--Hugoniot}
\end{equation}
holds when $s = \xi_\star$.

\begin{definition}
  \label{def:Rankine--Hugoniot}
  An \emph{R-H jump} is a triple $(u^-, s, u^+)$ with $u^+ \ne u^-$
  such that the
  Rankine--Hugoniot condition~\eqref{eq:Rankine--Hugoniot}
  holds.
  Here $u^-$ is the \emph{left state},
  $u^+$ is the \emph{right state},
  and $s$ is the \emph{propagation speed}.
\end{definition}

\section{Integral equation}
\label{section:integral_equation}

Dafermos~\cite{Daf08,Daf26}
has derived an integral equation satisfied by
self-similar weak solutions
that is a powerful tool for establishing structure and regularity of solutions.

\subsection{Dafermos function}
\label{sec:Dafermos_function}

\begin{definition}
  \label{def:Dafermos_function}
  Suppose $v$ is a reduced state function.
  With $\xi_0 \in \Rset$ arbitrary but fixed,
  the \emph{Dafermos function} $\mD \in L^\infty_\mathrm{loc}(\Rset;
  \Rset^n)$ is defined by
  \begin{equation}
    \mD(\xi) := -\xi\,v(\xi) + f\bigl(v(\xi)\bigr)
    + \int_{\xi_0}^\xi v(\xi')\,d\xi' \ \text{for almost every $\xi \in \Rset$.}
    \label{eq:Dafermos_function}
  \end{equation}
\end{definition}

A measurable function is called \emph{essentially constant}
if it equals a constant function a.e.

\begin{theorem}[Dafermos~\cite{Daf08,Daf26}]\label{th:Dafermos}
  Suppose that $v$ is a self-similar state function.
  Then $v$ is a weak solution of ~\eqref{eq:conslaw} if and only if
  $\mD$ is essentially constant.
\end{theorem}

\begin{proof}
  By Prop.~\ref{prop:Self-Similar_Weak_Condition},
  $v$ is a weak solution
  if and only if Eq.~\eqref{eq:self-similar_weak_condition}
  holds.
  Since $v \in L^\infty(\Rset; \Rset^n)$,
  the function
  \begin{equation}
    w(\xi) := \int_{\xi_0}^\xi v(\xi')\,d\xi'
    \label{eq:integral_term}
  \end{equation}
  is Lipschitz continuous and $d w / d\xi = v$ a.e.
  Let $\psi \in C^1_c(\Rset; \Rset^n)$.
  Then
  \begin{equation}
    -\int_{-\infty}^\infty \psi\cdot v\,d\xi
    = -\int_{-\infty}^\infty \psi\cdot \frac{d w}{d \xi}\,d\xi
    = \int_{-\infty}^\infty \frac{d\psi}{d\xi}\cdot w\,d\xi,
  \end{equation}
  where we have used integration by parts in the second step.
  The integral in condition~\eqref{eq:self-similar_weak_condition}
  can be related to an integral involving the Dafermos function as follows:
  \begin{equation}
    \int_{-\infty}^\infty
    \left\{\frac{d\psi}{d\xi}\cdot [-\xi\,v + f(v)]
    - \psi\cdot v\right\} d\xi = \int_{-\infty}^\infty
    \frac{d\psi}{d\xi}\cdot \left[-\xi\,v + f(v) + w\right] d\xi
    = \int_{-\infty}^\infty \frac{d\psi}{d\xi}\cdot \mD\,d\xi.
  \end{equation}
  Thus,
  condition~\eqref{eq:self-similar_weak_condition}
  holds if and only if
  \begin{equation}
    \int_{-\infty}^\infty \frac{d\psi}{d\xi}\cdot \mD\,d\xi = 0
    \ \text{for all $\psi \in C^1_c(\Rset; \Rset^n)$.}
    \label{eq:Dafermos_self-similar_weak_condition}
  \end{equation}
  If we view $\mD \in L^\infty_\mathrm{loc}(\Rset; \Rset^n)$ as a distribution,
  this equation says that the distributional derivative of $\mD$ vanishes.
  Equivalently, $\mD$ equals a constant $\mD_0$ a.e.~\cite[Theorem~3.4]{Car25}.
\end{proof}

\begin{corollary}[Dafermos~\cite{Daf08,Daf26}]
  \label{co:Dafermos}
  Suppose that $v$ is a reduced state function.
  If $v$ is a self-similar weak solution of system~\eqref{eq:conslaw},
  then
  \begin{equation}
    \text{$-\xi\,v + f(v) = \mF$ a.e.}
    \label{eq:moving_frame_flux_relation}
  \end{equation}
  with $\mF : \Rset \to \Rset^n$ being a Lipschitz continuous function.
\end{corollary}

\begin{proof}
  By Theorem~\ref{th:Dafermos},
  $v$ is a self-similar weak solution of system~\eqref{eq:conslaw}
  if and only if $\mD$ is essentially constant.
  If $\mD = \mD_0$ a.e.,
  define
  \begin{equation}
    \mF(\xi) := \mD_0 - \int_{\xi_0}^\xi v(\xi')\,d\xi'
  \end{equation}
  for all $\xi \in \Rset$.
  Then $-\xi\,v + f(v) = \mF$ a.e.
  Moreover,
  $\mF$ is Lipschitz continuous
  because $v \in L^\infty(\Rset; \Rset^n)$.
\end{proof}

\begin{definition}
  \label{def:moving_frame_flux}
  We refer to the Lipschitz continuous function $\mF$
  associated with a self-similar weak solution $v$
  of system~\eqref{eq:conslaw}
  as the \emph{moving frame flux} of $v$.
\end{definition}

\begin{remark}
  \label{rem:moving_frame_flux}
  To justify this terminology,
  let $u$ be the self-similar state function
  associated to $v$.
  For a fixed $\xi_\star \in \Rset$,
  replace the coordinate $x$ by $y := x - \xi_\star\,t$,
  \ie switch to the frame of reference moving at speed $\xi_\star$.
  Transformed to this frame,
  the solution $u$ becomes $U$, as defined by
  \begin{equation}
    U(y, t) := u(y + \xi_\star\,t, t) = v(y/t + \xi_\star)
  \end{equation}
  for a.e.\ $(y, t) \in \bbH$.
  Moreover,
  $U$ is a solution of the transformed system~\eqref{eq:conslaw}:
  with the flux function $f^{\xi_\star} : \bbU \to \Rset^n$ defined by
  \begin{equation}
    f^{\xi_\star}(U) := -\xi_\star\,U + f(U)
  \end{equation}
  for $U \in \bbU$,
  \begin{equation}
    U_t + f^{\xi_\star}(U)_y = 0.
  \end{equation}
  In other words, $f^{\xi_\star}$ is the flux function in the moving frame.
  We adopt Def.~\ref{def:moving_frame_flux} because
  \begin{equation}
    \text{$\mF(\xi) = f^{\xi_\star}\bigl(v(\xi)\bigr)$ for almost
    every $\xi \in \Rset$.}
    \label{eq:moving_frame_flux_relation_explicit}
  \end{equation}
  \qed
\end{remark}

\subsection{Integral equation}
\label{sec:integral_equation}

Theorem~\ref{th:Dafermos}
says that $v$ is a self-similar weak solution if and only if
\begin{equation}
  \mD(\xi) - \mD(\zeta) = 0 \ \text{for almost every $(\zeta, \xi)
  \in \Rset^2$.}
  \label{eq:D_xi_minus_D_zeta=0}
\end{equation}
We now cast this condition into a form that is more useful to us.

Recall our hypothesis that the convex hull of $\bbU_{SH}$ is contained in $\bbU$:
if $u^-$, $u^+ \in \bbU_{SH}$,
then $(1 - \tau)\,u^- + \tau\,u^+ \in \bbU$ for all $\tau \in [0, 1]$.
As $D f(u) = A(u)$, we find that
\begin{equation}
  \frac{d}{d\tau} f\bigl((1 - \tau)\,u^- + \tau\,u^+\bigr)
  = A\bigl((1 - \tau)\,u^- + \tau\,u^+\bigr)\,(u^+ - u^-).
  \label{eq:derivative}
\end{equation}
Let
\begin{equation}
  \overline{A}(u^-, u^+) := \int_0^1 A\bigl((1 - \tau)\,u^- +
  \tau\,u^+\bigr)\,d\tau
  \label{eq:overline_A}
\end{equation}
be the average of the characteristic matrix $A$ over the line from
$u^-$ to $u^+$.
Then integrating Eq.~\eqref{eq:derivative}
with respect to $\tau$ from $0$ to $1$ shows that
\begin{equation}
  f(u^+) - f(u^-) = \overline{A}(u^-, u^+)\,(u^+ - u^-).
  \label{eq:Taylor}
\end{equation}
Because the characteristic matrix $A$ is $C^1$,
the averaged characteristic matrix $\overline{A}$
is a $C^1$ function of $(u^-, u^+)$.
Note also that $\overline{A}(u, u) = A(u)$.

We introduce the matrix
\begin{equation}
  M(u^-, \xi, u^+) := -\xi\,I + \overline{A}(u^-, u^+).
  \label{eq:matrix_M}
\end{equation}
According to Eq.~\eqref{eq:Taylor}
the Rankine--Hugoniot condition~\eqref{eq:Rankine--Hugoniot}
reads
\begin{equation}
  M(u^-, s, u^+)\,(u^+ - u^-) = 0.
  \label{eq:R-H_with_matrix_M}
\end{equation}

\begin{lemma}[Dafermos~\cite{Daf08,Daf26}]
  \label{lem:matrix}
  For almost every $(\zeta, \xi) \in \Rset^2$,
  \begin{equation}
    \mD(\xi) - \mD(\zeta) =
    M\bigl(v(\zeta), \xi, v(\xi)\bigr) \left[v(\xi) - v(\zeta)\right]
    + \int_{\zeta}^{\xi}
    \left[v(\xi') - v(\zeta)\right] d\xi'.
    \label{eq:D_b_minus_D_a_modified}
  \end{equation}
\end{lemma}

\begin{proof}
  Written in terms of the notation $u^- := v(\zeta)$ and $u^+ := v(\xi)$,
  Eq.~\eqref{eq:Dafermos_function}
  reads
  \begin{equation}
    \mD(\xi) - \mD(\zeta)
    = -\xi\,u^+ + f(u^+) - \left[-\zeta\,u^- + f(u^-)\right]
    + \int_{\zeta}^{\xi} v(\xi')\,d\xi'.
  \end{equation}
  Adding the trivial equation
  $0 = (\xi - \zeta)\,u^- - \int_{\zeta}^{\xi} u^-\,d\xi'$
  yields
  \begin{equation}
    \mD(\xi) - \mD(\zeta)
    = -\xi\,(u^+ - u^-) + f(u^+) - f(u^-)
    + \int_{\zeta}^{\xi} [v(\xi') - u^-]\,d\xi'.
  \end{equation}
  By Eq.~\eqref{eq:Taylor},
  \begin{equation}
    \mD(\xi) - \mD(\zeta)
    = \left[-\xi\,I + \overline{A}(u^-, u^+)\right](u^+ - u^-)
    + \int_{\zeta}^{\xi} [v(\xi') - u^-]\,d\xi'.
  \end{equation}
  Using the definition~\eqref{eq:matrix_M}
  and substituting $u^- = v(\zeta)$ and $u^+ = v(\xi)$
  yields Eq.~\eqref{eq:D_b_minus_D_a_modified}.
\end{proof}

The following conjunction of Theorem~\ref{th:Dafermos}
and Lemma~\ref{lem:matrix}
figures prominently in the proofs of
Prop.~\ref{prop:discontinuities}
concerning discontinuities
and Prop.~\ref{prop:constancy}
concerning constant states.
A variant is the basis for the proof of
Theorem~\ref{th:rarefaction_intervals}
concerning rarefaction waves.

\begin{proposition}[Dafermos~\cite{Daf08,Daf26}]
  \label{prop:integral_equation}
  Suppose that $v$ is a self-similar state function.
  Then $v$ is a weak solution of system~\eqref{eq:conslaw}
  if and only if
  \begin{equation}
    M\bigl(v(\zeta), \xi, v(\xi)\bigr) \left[v(\xi) - v(\zeta)\right]
    + \int_{\zeta}^{\xi}
    \left[v(\xi') - v(\zeta)\right] d\xi' = 0
    \ \text{for almost every $(\zeta, \xi) \in \Rset^2$.}
    \label{eq:integral_equation}
  \end{equation}
\end{proposition}

\section{Continuity}
\label{section:continuity}

\subsection{Continuity for weak solutions}
\label{sec:continuity_for_weak_solutions}

Suppose that $v$ is a self-similar weak solution
of system~\eqref{eq:conslaw}.
Traditionally,
$v$ is assumed to be piecewise smooth
and then is assembled from smooth rarefaction waves,
jump discontinuities that satisfy the Rankine--Hugoniot condition,
and constant states.
The speed axis $\xi = x/t$ has a corresponding partition into sets
where $v$ varies continuously, has discontinuities, or is constant.
We seek such a partition when $v$ is only assumed to be essentially bounded.

Defining this partition for members of $L^\infty$ is challenging.
Measurable functions are regarded as equivalent
if they agree except on a set of measure zero.
A member of $L^\infty(\Rset, \Rset^n)$
is an equivalence class of functions that agree a.e.;
its values at particular points $\xi \in \Rset$
and on sequences of points $\{ \xi_k \}_{k=1}^\infty$ in $\Rset$
are arbitrary.

One may define $\xi$ to be a point of continuity for a member of $L^\infty$
if there exists a representative that is continuous at $\xi$.
Then discontinuity would be defined only by the absence of a
continuous representative,
not by the existence of distinct one-sided accumulation states,
which are indispensable for verifying the Rankine--Hugoniot condition.

Another way of addressing this problem is
to employ \emph{approximate continuity}~\cite[\S~1.7]{EvaGar15}.
However, this approach is inadequate for our purposes because
$v$ can be approximately continuous at every point
even though no continuous function $\overline{v}$ exists
with $v = \overline{v}$ a.e.
(A Volterra-type ``thin spikes'' function~\cite[Chapter~3]{Bru94}
provides an example.)

Instead, we adopt notions of continuity and accumulation sets
based on the concept of \emph{essential image}.
Appendix~\ref{sec:essential_image}
provides definitions of \ESSIM{} accumulation sets
and \ESSIM{} continuity for a measurable function $h : \Rset^p \to \Rset^n$
and proves the basic results we invoke.
In the present section,
we specialize to the case relevant to self-similar weak solutions:
the domain is one-dimensional and we can distinguish one-sided
accumulation sets.

Let $J \subseteq \Rset$ be measurable.
By Defs.~\ref{def:essential_image}
and~\ref{def:essential_image_of_subset},
the essential image of $J$ under $v$,
denoted
\begin{equation}
  \essim v\restrict_J,
\end{equation}
is the set of $u \in \Rset^n$ such that $v^{-1}[\Uopen] \cap J$
has positive Lebesgue measure for each neighborhood $\Uopen$ of $u$.
Notice that the essential image is invariant under equality a.e.,
so that it is well-defined on $L^\infty$.

Lemma~\ref{lem:monotonicity}
and Lemma~\ref{lem:essential_image}
justify Def.~\ref{def:essential_accumulation_sets}:
the \emph{\ESSIM{} accumulation set} of $v$ at $\xi \in \Rset$ is
the intersection of the essential images of intervals shrinking around $\xi$:
\begin{equation}
  \Aessim(v; \xi) := \bigcap_{r > 0} \essim v\restrict_{(\xi - r,\,\xi + r)}.
\end{equation}
For the same reason, we can define one-sided accumulation sets.

\begin{definition}
  \label{def:essential-image_accumulation_set}
  The \emph{left} and \emph{right \ESSIM{} accumulation sets} are
  the intersections of the essential images of shrinking one-sided intervals:
  \begin{equation}
    \Aessim^-(v; \xi) := \bigcap_{r > 0} \essim v\restrict_{(\xi - r,\,\xi)}
    \quad \text{and} \quad
    \Aessim^+(v; \xi) := \bigcap_{r > 0} \essim v\restrict_{(\xi,\,\xi + r)}.
  \end{equation}
  Elements of these sets are the \emph{\ESSIM{} accumulation states}
  of $v$ at $\xi$.
\end{definition}

The same argument as in the proof of Lemma~\ref{lem:inclusion}
shows that $\Aessim^-(v; \xi)$ and $\Aessim^+(v; \xi)$ are nonempty and compact.
As $(\xi - r, \xi + r) = (\xi - r, \xi) \cup (\xi, \xi + r) \cup \{\xi\}$,
Lemma~\ref{lem:union} and Lemma~\ref{lem:monotonicity}
imply that
\begin{equation}
  \Aessim(v; \xi) = \Aessim^-(v; \xi) \cup \Aessim^+(v; \xi).
  \label{eq:Aessim_union}
\end{equation}
By Def.~\ref{def:essential-image_continuity},
the point $\xi$ is an \emph{\ESSIM{} continuity point for $v$}
if and only if $\Aessim^-(v; \xi)$
and $\Aessim^+(v; \xi)$ are the same singleton (one-point set);
in this case,
we let $\overline{v}(\xi) \in \Rset^n$ denote
the \emph{\ESSIM{} accumulation value of $v$ at $\xi$}:
\begin{equation}
  \Aessim^-(v; \xi) = \bigl\{\overline{v}(\xi)\bigr\} = \Aessim^+(v; \xi).
  \label{eq:essential-image_continuity_point_state_function}
\end{equation}
Otherwise, the point $\xi$ is an \emph{\ESSIM{} discontinuity point for $v$}.

\begin{definition}
  \label{def:essential-image_jump_discontinuity}
  A point $\xi$ is an \emph{\ESSIM{} jump discontinuity point of $v$}
  if $\Aessim^-(v; \xi)$ and $\Aessim^+(v; \xi)$ are distinct singletons;
  in this case,
  we let $v^-(\xi)$ and $v^+(\xi)$ denote
  the \emph{\ESSIM{} left and right accumulation values of $v$ at $\xi$}:
  \begin{equation}
    \text{$\Aessim^-(v; \xi) = \bigl\{v^-(\xi)\bigr\}$
      and $\Aessim^+(v; \xi) = \bigl\{v^+(\xi)\bigr\}$
    with $v^+(\xi) \ne v^-(\xi)$.}
    \label{eq:essential-image_jump_discontinuity_point_state_function}
  \end{equation}
\end{definition}

The following properties of \ESSIM{} continuity
and \ESSIM{} accumulation sets
are central to our application to the theory of conservation laws.
Lemma~\ref{lem:continuity_lemma}
entails the first result,
which allows us to turn an a.e.\ relationship into a pointwise one
on intervals of \ESSIM{} continuity.

\begin{lemma}
  \label{lem:continuity_lemma_state_function}
  Let $v \in L^\infty(\Rset; \Rset^n)$.
  Suppose every point in an open interval $J \subseteq \Rset$
  is an \ESSIM{} continuity point for $v$.
  Then $\overline{v} : J \to \Rset^n$ is continuous
  and $\overline{v} = v$ a.e.\ in $J$.
\end{lemma}

As a corollary,
when $v$ is \ESSIM{} continuous in an open interval $J$,
we may replace $v$ by $\overline{v}$ in
Eq.~\eqref{eq:moving_frame_flux_relation}
to find that $-\xi\,\overline{v}(\xi) +
f\bigl(\overline{v}(\xi)\bigr) = \mF(\xi)$
for a.e. $\xi \in J$.
By continuity, this equality holds throughout $J$:
\begin{equation}
  -\xi\,\overline{v}(\xi) + f\bigl(\overline{v}(\xi)\bigr) = \mF(\xi)
  \ \text{for all $\xi \in J$.}
  \label{eq:moving_frame_flux_relation_with_overline_v}
\end{equation}
Similarly,
consider replacing $v$ by $\overline{v}$
in the integral equation~\eqref{eq:integral_equation}.
The left-hand side of the resulting equation
is a continuous function of $(\xi, \zeta) \in J \times J$ that vanishes a.e.,
so we draw the following conclusion.

\begin{lemma}
  \label{lem:everywhere_lemma}
  Let $v$ be a self-similar weak solution
  of system~\eqref{eq:conslaw}.
  Suppose every point in an open interval $J \subseteq \Rset$
  is an \ESSIM{} continuity point for $v$.
  Then
  \begin{equation}
    M\bigl(\overline{v}(\zeta), \xi, \overline{v}(\xi)\bigr)
    \left[\overline{v}(\xi) - \overline{v}(\zeta)\right]
    + \int_{\zeta}^{\xi}
    \left[\overline{v}(\xi') - \overline{v}(\zeta)\right] d\xi' = 0
    \ \text{for all $\zeta$, $\xi \in J$.}
    \label{eq:integral_equation_with_overline_v}
  \end{equation}
\end{lemma}

\begin{remark}
  An analogous integral equation employed in the proof of
  Theorem~\ref{th:rarefaction_intervals}
  likewise holds pointwise on intervals of \ESSIM{} continuity.
  \qed
\end{remark}

A related equation constrains the possible states $u \in \Aessim(v; \xi_\star)$
at any $\xi_\star \in \Rset$.

\begin{lemma}
  \label{lem:discontinuity_lemma}
  Let $v$ be a self-similar weak solution
  of system~\eqref{eq:conslaw},
  and let $\mF$ be the moving frame flux of $v$.
  If $\xi_\star \in \Rset$ and $u \in \Aessim(v; \xi_\star)$,
  then
  \begin{equation}
    -\xi_\star\,u + f(u) = \mF(\xi_\star).
    \label{eq:integral_term_with_u}
  \end{equation}
\end{lemma}

\begin{proof}
  By Theorem~\ref{th:Dafermos},
  $\mF$ is Lipschitz continuous
  and $-\xi\,v(\xi) + f\bigl(v(\xi)\bigr) - \mF(\xi) = 0$ for a.e.
  $\xi \in \Rset$.
  Define $\Psi : \Rset \times \bbU \to \Rset^n$
  by $\Psi(\xi, u) := -\xi\,u + f(u) - \mF(\xi)$.
  Lemma~\ref{lem:ae_to_Aessim}
  applied to $h = v$ and $\Psi$ yields the result.
\end{proof}

Equation~\eqref{eq:integral_term_with_u}
has the following interpretation
when $\xi_\star$ is an \ESSIM{} discontinuity point.
Pick one state $u_0 \in \Aessim(v; \xi_\star)$.
Then any other state $u \in \Aessim(v; \xi_\star)$ satisfies
\begin{equation}
  -\xi_\star\,(u - u_0) + f(u) - f(u_0) = 0.
  \label{eq:R-H_at_discontinuity}
\end{equation}
Thus, $u$ lies on the Hugoniot locus of $u_0$ and has propagation
speed $\xi_\star$.

\section{Constant states}
\label{section:constant_states}

\subsection{Constant states}
\label{sec:constant_states}

A state function satisfies the system of conservation laws
trivially wherever it is constant.
The next result gives a condition forcing constancy:
if the solution is continuous and \emph{non-resonant} in an open interval,
then it is constant there.

\begin{proposition}[Dafermos~\cite{Daf08,Daf26}]
  \label{prop:constancy}
  Suppose that every point in an open interval $J \subseteq \Rset$
  is an \ESSIM{} continuity point
  for the self-similar weak solution $v$.
  Let $\overline{v}$ denote the continuous representative of $v$ in $J$.
  If $\overline{v}$ is \emph{non-resonant} in $J$, meaning
  \begin{equation}
    \text{$\det\left[-\xi\,I + A\bigl(\overline{v}(\xi)\bigr)\right] \ne 0$
    for all $\xi \in J$,}
    \label{eq:non-resonant}
  \end{equation}
  then $\overline{v}$ is constant in $J$.
\end{proposition}

\begin{proof}
  Fix $\zeta \in J$.
  Because $\overline{v}$ is continuous,
  \begin{equation}
    \int_{\zeta}^{\xi}
    \left[\overline{v}(\xi') - \overline{v}(\zeta)\right] d\xi' = o(\abs{\xi - \zeta})
    \ \text{as $\xi \to \zeta$,}
  \end{equation}
  and Eqs.~\eqref{eq:matrix_M}
  and~\eqref{eq:overline_A}
  imply that
  \begin{equation}
    M\bigl(\overline{v}(\zeta), \xi, \overline{v}(\xi)\bigr)
    = -\zeta\,I + A\bigl(\overline{v}(\zeta)\bigr) + o(1)
    \ \text{as $\xi \to \zeta$}.
  \end{equation}
  Therefore,
  Eq.~\eqref{eq:integral_equation_with_overline_v}
  reads
  \begin{equation}
    \left[-\zeta\,I + A\bigl(\overline{v}(\zeta)\bigr) + o(1)\right]
    \left[\overline{v}(\xi) - \overline{v}(\zeta)\right]
    + o(\abs{\xi - \zeta}) = 0
    \ \text{as $\xi \to \zeta$}.
  \end{equation}
  The non-resonance hypothesis~\eqref{eq:non-resonant}
  entails that the matrix in the first bracket is invertible
  for $\xi$ sufficiently close to $\zeta$.
  Hence,
  \begin{equation}
    \overline{v}(\xi) - \overline{v}(\zeta) = o(\abs{\xi - \zeta})
    \ \text{as $\xi \to \zeta$.}
  \end{equation}
  In other words,
  \begin{equation}
    \frac{\overline{v}(\xi) - \overline{v}(\zeta)}{\xi - \zeta} \to 0
    \ \text{as $\xi \to \zeta$,}
  \end{equation}
  \ie $\overline{v}$ is differentiable,
  and its derivative is $0$, at $\zeta$.
  As $\zeta \in J$ is arbitrary,
  $\overline{v}$ has vanishing derivative throughout the open interval $J$.
  Thus, $\overline{v}$ is constant in $J$.
\end{proof}

\section{Rarefaction waves}
\label{section:rarefaction_waves}

In the strictly hyperbolic region $\bbU_{SH} \subseteq \bbU$ of state space,
the characteristic matrix $A$ has real, distinct eigenvalues.
Therefore, there exist $C^1$ functions, denoted
$\lambda_i : \bbU_{SH} \to \Rset$ for $i = 1$, $\ldots$, $n$,
such that each $\lambda_i(u)$ is an eigenvalue of $A(u)$
and
\begin{equation}
  \lambda_1(u) < \cdots < \lambda_n(u)
\end{equation}
for all $u \in \bbU_{SH}$.
The index $i$ is called the \emph{eigenvalue family}.
We focus on a particular characteristic family in a neighborhood of a particular state
and construct a corresponding smooth right eigenvector field.
To simplify notation, we omit the family index.

\subsection{Characteristic speed}
\label{sec:characteristic_speed}

If $u \in \bbU_{SH}$, then $\lambda \in \Rset$
is a \emph{characteristic speed} for system~\eqref{eq:conslaw}
at $u$ if it is a root of the \emph{characteristic polynomial}
for the characteristic matrix $A(u)$:
\begin{equation}
  p\bigl(\lambda, A(u)\bigr) := \det\left[-\lambda\,I + A(u)\right].
  \label{eq:characteristic_speed}
\end{equation}
By strict hyperbolicity, $\lambda$ is a simple root.
A \emph{right characteristic eigenvector} corresponding to $\lambda$
is a nonzero vector $r \in \Rset^n$ such that
\begin{equation}
  \left[-\lambda\,I + A(u)\right] r = 0.
  \label{eq:right_eigenvector}
\end{equation}

A \emph{characteristic speed map} is a $C^1$ function $\lambda :
\Uopen \to \Rset$,
where $\Uopen \subseteq \bbU_{SH}$ is open,
such that $\lambda(u)$ is an eigenvalue of $A(u)$ for all $u \in \Uopen$.
The following local construction provides
the characteristic speed map $\lambda$ and corresponding eigenvector field $r$.

\begin{lemma}
  \label{lem:local_characteristic_speed_map_and_eigenvector_field}
  Let $u_\star \in \bbU_{SH}$.
  Suppose that $\lambda_\star \in \Rset$
  is a root of the characteristic polynomial at $u_\star$.
  Then there exist a neighborhood $\Uopen \subseteq \bbU_{SH}$ of $u_\star$,
  a $C^1$ function $\lambda : \Uopen \to \Rset$ with
  $\lambda(u_\star) = \lambda_\star$,
  and a $C^1$ map $r : \Uopen \to \Rset^n\,\setminus\,\{0\}$
  such that, for all $u \in \Uopen$,
  \begin{equation}
    \left[-\lambda(u)\,I + A(u)\right] r(u) = 0.
    \label{eq:local_eigenvector}
  \end{equation}
\end{lemma}

\begin{proof}
  Let $M_\star := -\lambda_\star\,I + A(u_\star)$.
  The derivative of the characteristic polynomial at $\lambda_\star$ is nonzero
  because the root $\lambda_\star$ is simple.
  This derivative is
  $p_\lambda(\lambda_\star, A(u_\star)) = \tr \adj M_\star$,
  which is nonzero if and only if $M_\star$ has rank $n - 1$
  and there exist left and right null vectors $\ell_\star$ and $r_\star$
  of $M_\star$ such that $\ell_\star\,r_\star = 1$.
  Consider the $C^1$ system of equations
  for $\lambda \in \Rset$ and $r \in \Rset^n$
  comprising Eq.~\eqref{eq:right_eigenvector}
  and $\ell_\star\,r = 1$.
  We obtain a local $C^1$ solution by applying the implicit function theorem
  at the point $(\lambda, r, u) = (\lambda_\star, r_\star, u_\star)$.
  To verify its hypotheses,
  we show that $(\dot{\lambda}, \dot{r}) = (0, 0)$
  is the only solution of the linear system
  \begin{align}
    -\dot{\lambda}\,r_\star + M_\star\,\dot{r} & = 0, \\
    \ell_\star\,\dot{r}                        & = 0
  \end{align}
  for $(\dot{\lambda}, \dot{r}) \in \Rset \times \Rset^n$.
  Applying $\ell_\star$ to the first equation shows that $\dot{\lambda} = 0$;
  then $M_\star\,\dot{r} = 0$ requires $\dot{r}$ to be a multiple of $r_\star$,
  hence the second equation forces $\dot{r} = 0$.
\end{proof}

A \emph{rarefaction integral curve} for the characteristic speed map $\lambda$
is a $C^1$ curve $\Uinteg : I \to \Uopen$,
where $I \subseteq \Rset$ is a nonempty open interval,
that solves the ordinary differential equation
\begin{equation}
  \frac{d \Uinteg}{d\eta} = r\bigl(\Uinteg(\eta)\bigr)
  \label{eq:rarefaction_locus_ODE}
\end{equation}
for $\eta \in I$.
The existence of a local $C^1$ flow for the $C^1$ vector field $r$,
which we state in the following lemma,
is standard.

\begin{lemma}
  \label{lem:local_rarefaction_flow}
  In the context of
  Lemma~\textup{\ref{lem:local_characteristic_speed_map_and_eigenvector_field}},
  there exist a neighborhood $\Uopen' \subseteq \Uopen$ of $u_\star$,
  a nonempty symmetric open interval $I \subseteq \Rset$,
  and a $C^1$ map $\Uinteg : I \times \Uopen' \to \Uopen$ such that
  \begin{align}
    \Uinteg_\eta(\eta, u) & = r\bigl(\Uinteg(\eta, u)\bigr),
    \label{eq:flow_ODE}                                      \\
    \Uinteg(0, u)         & = u
    \label{eq:flow_IV}
  \end{align}
  for all $\eta \in I$ and $u \in \Uopen'$.
  If $u$ is held fixed, $\eta \mapsto \Uinteg(\eta, u)$
  is the unique rarefaction integral curve through $u$.
  Moreover, for each fixed $\eta \in I$,
  the map $\Uinteg(\eta, \cdot)$ is a $C^1$ diffeomorphism
  between $\Uopen'$ and its image,
  the inverse being $\Uinteg(-\eta, \cdot)$.
\end{lemma}

As a corollary, there is a coordinate system near $u_\star$
in which the rarefaction integral curves corresponding to $\lambda$
are parallel straight lines.
See Fig.~\ref{fig:straighten}
for an illustration when $n = 2$.
This coordinate system plays a crucial role in the proof of
Theorem~\ref{th:rarefaction_intervals}.

\begin{figure}
  \centering
  \includegraphics[width=\textwidth]{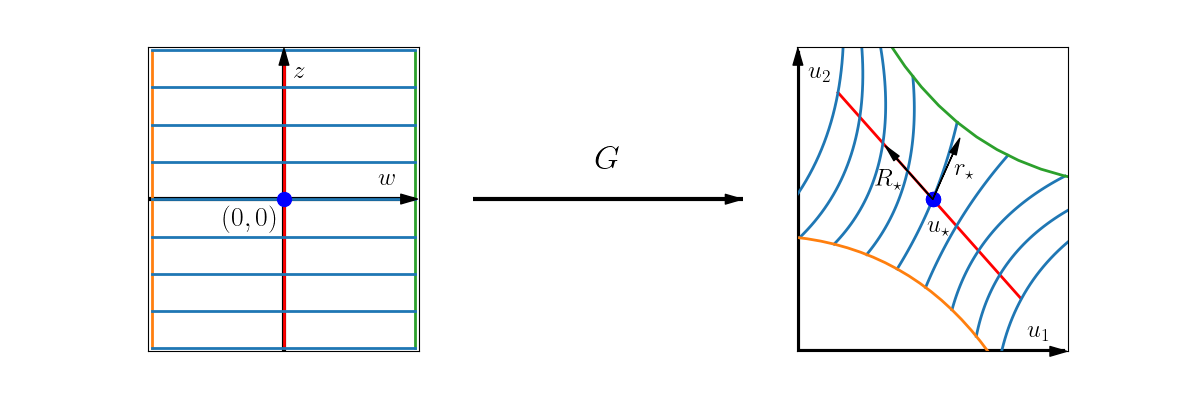}
  \caption[Straightened rarefaction integral curves.]{
    There exist $(w, z)$-coordinates for a neighborhood of $u_\star$
    in which rarefaction integral curves are straight.
  }
  \label{fig:straighten}
\end{figure}

\begin{corollary}[Straightening Rarefaction Integral Curves]
  \label{cor:straightening_rarefaction_integral_curves}
  Let $u_\star \in \bbU_{SH}$, and suppose that $\lambda_\star \in \Rset$
  is a root of the characteristic polynomial at $u_\star$.
  Then there exist an open interval $I \subseteq \Rset$ containing $0$,
  an open ball $B \subseteq \Rset^{n-1}$ centered at $0$,
  an open neighborhood $\Uopen \subseteq \bbU_{SH}$ of $u_\star$,
  and a $C^1$ diffeomorphism $G : I \times B \to \Uopen$
  such that $G(0,0) = u_\star$ and, for each fixed $z \in B$,
  the curve $w \mapsto G(w, z)$ for $w \in I$ is a rarefaction integral curve.
  Explicitly, with $F := f\comp G$ and $\mu := \lambda\comp G$,
  \begin{align}
    & \text{$-\mu\,G_w + F_w = 0$ throughout $I \times B$ and}
    \label{eq:partial_w_corollary}                              \\
    & \text{$(-\mu\,G_z + F_z)(0, 0)$ has rank $n - 1$.}
    \label{eq:partial_z_corollary}
  \end{align}
\end{corollary}

\begin{proof}
  Invoke Lemma~\ref{lem:local_characteristic_speed_map_and_eigenvector_field}.
  Choose a basis for $\Rset^n$ that includes $r_\star$,
  and let the columns of the $n \times (n-1)$ matrix $R_\star$
  be the other $n - 1$ basis elements.
  Let the open ball $B \subseteq \Rset^{n-1}$ centered at $0$ be such that
  $u_\star + R_\star\,z \in \Uopen'$ if $z \in B$.
  Define $G : I \times B \to \Uopen$ by
  \begin{equation}
    G(w, z) := \Uinteg(w, u_\star + R_\star\,z)
    \label{eq:G_definition}
  \end{equation}
  for $(w, z) \in I \times B$.
  Evaluated at $(w, z) = (0, 0)$,
  $G_w = r_\star$ and, by Eq.~\eqref{eq:flow_IV},
  $G_z = R_\star$.
  Hence, $D G(0, 0) = (r_\star, R_\star)$ is invertible.
  After possibly shrinking $I$ and $B$
  and redefining $\Uopen$ to be the image of $G$,
  we obtain the diffeomorphism $G$.

  Now, $D F = D(f\comp G) = \left[(D f)\comp G\right] D G$, so that
  \begin{equation}
    -\mu\,D G + D F = \left[(-\lambda\,I + D f)\comp G\right] D G.
    \label{eq:mu_DG_plus_DF}
  \end{equation}
  Differentiating definition~\eqref{eq:G_definition}
  with respect to $w$ gives $G_w = \Uinteg_\eta = r\comp G$.
  Therefore, Eq.~\eqref{eq:mu_DG_plus_DF}
  implies that
  \begin{equation}
    -\mu\,G_w + F_w = \left[(-\lambda\,I + D f)\comp G\right] r\comp G = 0.
  \end{equation}
  Also,
  \begin{align}
    \left(-\mu\,G_z + F_z\right)(0, 0)
    & = \left[-\lambda\,I + D f(u_\star)\right] R_\star
  \end{align}
  has rank $n - 1$.
\end{proof}

\begin{remark}
  \label{rem:Riemann_invariants}
  We emphasize that
  Corollary~\ref{cor:straightening_rarefaction_integral_curves}
  does not rely on the existence of a coordinate system of $n$
  Riemann invariants
  for system~\eqref{eq:conslaw}.
  We straighten rarefaction integral curves for only one characteristic speed.
  \qed
\end{remark}

\subsection{Locally transformed solution}
\label{sec:locally_transformed_solution}

The $C^1$ diffeomorphism
of Corollary~\ref{cor:straightening_rarefaction_integral_curves}
transforms the system of conservation laws~\eqref{eq:conslaw}
into a form with a special property,
which we now highlight.
This property is utilized in the proof of
Theorem~\ref{th:rarefaction_intervals}.

Consider a nonlinear change of variables
\begin{equation}
  u = G(U).
  \label{eq:transform}
\end{equation}
Corresponding to the flux $f$ is
\begin{equation}
  F := f\comp G,
\end{equation}
so that
\begin{equation}
  f(u) = F(U)
\end{equation}
when relationship~\eqref{eq:transform}
holds.
In a sense,
the PDE~\eqref{eq:conslaw}
becomes
\begin{equation}
  G(U)_t + F(U)_x = 0.
\end{equation}

To be precise,
suppose that $u : \bbH \to \bbU$
is a weak solution of system~\eqref{eq:conslaw}:
$u \in L^\infty(\bbH; \Rset^n)$ and
\begin{equation}
  \int_0^\infty \int_{-\infty}^\infty
  \left[\phi_t\cdot u + \phi_x\cdot f(u)\right] d x\,d t = 0
  \ \text{for all $\phi \in C^1_c(\bbH; \Rset^n)$.}
  \label{eq:weak_condition_repeat}
\end{equation}
Also assume that every point in a bounded open subset $\Omega \subseteq \bbH$
is an \ESSIM{} continuity point for $u$.
By Lemma~\ref{lem:continuity_lemma},
$\overline{u}$ is continuous in $\Omega$
and $\overline{u} = u$ a.e.\ in $\Omega$.
Fix $(x_\star, t_\star) \in \Omega$
and set $u_\star := \overline{u}(x_\star, t_\star) \in \bbU_{SH}$.

Apply Corollary~\ref{cor:straightening_rarefaction_integral_curves}
to find the $C^1$ diffeomorphism $G : I \times B \to \Uopen$,
where $\Uopen$ is an open neighborhood of $u_\star$ in $\bbU_{SH}$.
Replace $\Omega$ by
$\overline{u}^{\,-1}[\Uopen]$,
which is open because $\overline{u}$ is continuous.
Define $U : \Omega \to I \times B$ through
\begin{equation}
  U := G^{-1}\comp\left(\overline{u}\restrict_\Omega\right).
\end{equation}
In condition~\eqref{eq:weak_condition_repeat},
with $u$ replaced by $\overline{u}$,
restrict to test functions $\phi$ with support contained in $\Omega$.
Then this condition becomes
\begin{equation}
  \int_0^\infty \int_{-\infty}^\infty
  \left[\phi_t\cdot G(U) + \phi_x\cdot F(U)\right] d x\,d t = 0
  \ \text{for all $\phi \in C^1_c(\Omega; \Rset^n)$.}
  \label{eq:weak_condition_transformed}
\end{equation}

Similarly, suppose that $v$ is a self-similar weak solution
and that every point in an open interval $J \subseteq \Rset$
is an \ESSIM{} continuity point for $v$.
By Lemma~\ref{lem:continuity_lemma_state_function},
$\overline{v} : J \to \Rset^n$ is continuous in $J$
and $\overline{v} = v$ a.e.\ in $J$.
Fix $\xi_\star \in J$
and set $u_\star := \overline{v}(\xi_\star)$.
Replace $J$ by $\overline{v}^{\,-1}[\Uopen]$ and define
\begin{equation}
  V := (w, z) := G^{-1}\comp\left(\overline{v}\restrict_J\right).
\end{equation}
Then condition~\eqref{eq:self-similar_weak_condition},
with $v$ replaced by $\overline{v}$ and $\psi$ restricted to have
support contained in $J$,
becomes
\begin{equation}
  \int_{-\infty}^\infty
  \left[-\frac{d}{d\xi}(\xi\,\psi)\cdot G(V)
  + \frac{d\psi}{d\xi}\cdot F(V)\right] d\xi = 0
  \ \text{for all $\psi \in C^1_c(J; \Rset^n)$.}
  \label{eq:self-similar_weak_condition_transformed}
\end{equation}

The development in Sec.~\ref{sec:Dafermos_function}
shows that this condition holds if and only if the Dafermos function
\begin{equation}
  \mD(\xi) := (-\xi\,G + F)\bigl(V(\xi)\bigr)
  + \int_{\xi_0}^\xi G\bigl(V(\xi')\bigr)\,d\xi'
  \label{eq:Dafermos_function_transformed}
\end{equation}
is such that
\begin{equation}
  \mD(\xi) - \mD(\zeta) = 0 \ \text{for all $\zeta$, $\xi \in J$.}
  \label{eq:D_xi_minus_D_zeta=0_transformed}
\end{equation}

\subsection{Smooth rarefaction waves}
\label{sec:smooth_rarefaction_waves}
In Sec.~\ref{sec:features_in_solutions},
we saw that one type of continuous self-similar solution
is a smooth rarefaction wave.
For such a solution, $d v/d \xi$ is an eigenvector of $A(v)$
with corresponding eigenvalue $\xi$.
In terms of the characteristic speed function $\lambda$
and local eigenvector field $r$ constructed in
Lemma~\ref{lem:local_characteristic_speed_map_and_eigenvector_field},
\begin{equation}
  \frac{d v}{d\xi} = \kappa\,r(v)
  \label{eq:propto}
\end{equation}
with the proportionality factor $\kappa$ determined by the resonance condition
\begin{equation}
  \xi = \lambda\bigl(v(\xi)\bigr).
  \label{eq:resonance_repeated}
\end{equation}
Differentiating Eq.~\eqref{eq:resonance_repeated}
shows that
\begin{equation}
  D\lambda(v)\,\frac{d v}{d\xi} = 1.
  \label{eq:derivative_of_resonance_repeated}
\end{equation}
Combining Eqs.~\eqref{eq:propto}
and~\eqref{eq:derivative_of_resonance_repeated},
we obtain
\begin{equation}
  \frac{d v}{d\xi} = \left[D\lambda(v)\,r(v)\right]^{-1} r(v).
  \label{eq:rarefaction_wave_ODE}
\end{equation}
Because it only makes sense when $D\lambda(v)\,r(v) \ne 0$,
this formula motivates the following definition.

\begin{definition}
  \label{def:genuinely_nonlinear}
  In the context of
  Lemma~\ref{lem:local_characteristic_speed_map_and_eigenvector_field},
  the characteristic speed function $\lambda$ is said to be
  \emph{genuinely nonlinear} at the state $u_\star \in \bbU_{SH}$ if
  \begin{equation}
    D\lambda(u_\star)\,r(u_\star) \ne 0.
    \label{eq:genuinely_nonlinear}
  \end{equation}
\end{definition}
\par\noindent
Under the assumption of genuine nonlinearity at the state $u_*$,
the ODE Eq.~\eqref{eq:rarefaction_wave_ODE}
for a smooth rarefaction wave is valid near $u_\star$.

\subsection{General rarefaction waves}
\label{sec:general_rarefaction_waves}
The requirement of genuine nonlinearity in rarefaction waves is unduly restrictive.
For example,
consider the scalar conservation law with flux function $f(u) = u^4 / 4$.
As $\lambda(u) = u^3$ and $D\lambda(u) = 3\,u^2$,
$\lambda$ is not genuinely nonlinear at $u = 0$.
Nonetheless,
the PDE
\begin{equation}
  u_t + \left(\tfrac{1}{4}\,u^4\right)_x = 0
\end{equation}
has a continuous self-similar weak solution, \viz
\begin{equation}
  v(\xi) =
  \begin{cases}
    -1                         & \text{if $\xi \le -1$,}   \\
    (\sgn \xi) \abs{\xi}^{1/3} & \text{if $-1 < \xi < 1$,} \\
    \phantom{-}1               & \text{if $\xi \ge 1$}
  \end{cases}
  \label{eq:u^4_rarefaction}
\end{equation}
that satisfies the resonance condition for $\xi \in (-1, 1)$.
We regard solution~\eqref{eq:u^4_rarefaction}
as a rarefaction wave in this interval,
even though it fails to be differentiable at $\xi = 0$.
Similar solutions occur in systems of conservation laws
when a rarefaction integral curve is tangent to the
set of states where genuine nonlinearity fails.

The fundamental features of a rarefaction wave are
(1)~$v$ lies along integral curve for $\lambda$ and (2)~$v$ is resonant.
Feature~(1) means that the map $\xi \mapsto v(\xi)$ is a reparametrization
of a smooth integral curve $\eta \mapsto \Uinteg(\eta)$ in terms of
speed $\xi$:
\begin{equation}
  v(\xi) := \Uinteg\bigl(w(\xi)\bigr),
\end{equation}
where $w$ maps an open interval $J$ continuously onto
an open interval $I$ within the domain of $\Uinteg$.
With $\widehat{\lambda} := \lambda\comp\Uinteg$
denoting the \emph{characteristic speed function along $\Uinteg$},
$\lambda \comp v = \widehat{\lambda} \comp w$,
so that feature~(2),
\viz the resonance condition~\eqref{eq:resonance_repeated},
entails
\begin{equation}
  \xi = \widehat{\lambda}\bigl(w(\xi)\bigr)
  \label{eq:resonance_again}
\end{equation}
for all $\xi \in J$.

\begin{lemma}
  \label{lem:rarefaction_wave}
  Let $\Uinteg$ be a rarefaction integral curve.
  There exists a reparametrization of $\Uinteg$ by speed $\xi$
  if and only if the characteristic speed function along $\Uinteg$
  is strictly monotone.
\end{lemma}

\begin{proof}
  By Eq.~\eqref{eq:resonance_again},
  $w$ is necessarily injective
  and $\widehat{\lambda}$ is surjective onto $J$.
  Therefore, $w$ is strictly monotone, hence a homeomorphism,
  and $\widehat{\lambda}$, viewed as a map $I \to J$,
  is its inverse, hence is strictly monotone.
  Conversely,
  if $\widehat{\lambda}$ is strictly monotone,
  it is a homeomorphism from $I$ to its image $J$,
  and $w := \widehat{\lambda}^{-1}$
  serves to reparametrize $\Uinteg$ by $\xi$.
\end{proof}

Motivated by this lemma as well as by features~(1) and~(2) of rarefaction waves,
we adopt the following broad definition of rarefaction wave.

\begin{definition}
  \label{def:rarefaction_wave}
  Suppose that $\Uopen \subseteq \bbU_{SH}$ is open
  and $\lambda : \Uopen \to \Rset$ is a $C^1$ function
  such that $\lambda(u)$ is an eigenvalue of $A(u)$ for all $u \in \Uopen$.
  A \emph{rarefaction wave} is a continuous curve $v : J \to \Uopen$ of the form
  \begin{equation}
    v = \Uinteg\comp w,
  \end{equation}
  where
  \begin{itemize}
    \item $\Uinteg : I \to \Uopen$
      is a rarefaction integral curve for $\lambda$
      for a nonempty open interval $I \subseteq \Rset$;
    \item $w : J \to I$ is a homeomorphism
      between a nonempty open interval $J \subseteq \Rset$ and $I$; and
    \item $\lambda\bigl(v(\xi)\bigr) = \xi$ for all $\xi \in J$.
  \end{itemize}
  In this case, we say that \emph{$v$ lies along $\Uinteg$}
  and $w$ \emph{reparametrizes $\Uinteg$ by speed $\xi$}.
\end{definition}

Notably, $w$ is not required to be differentiable.
Although the images of the rarefaction integral curve $\Uinteg$
and the rarefaction wave $v = \Uinteg\comp w$ coincide,
the function $v$ is not necessarily smooth.
For instance,
genuine nonlinearity~\eqref{eq:genuinely_nonlinear}
can fail along a rarefaction wave, as illustrated above.
On the other hand,
$v$ does display significant regularity:
it is continuous and has bounded variation
(because $w$ is a monotone function).

Definition~\ref{def:rarefaction_wave}
is justified by the following result.

\begin{proposition}
  \label{prop:rarefaction}
  If $v : J \to \Uopen$ is a rarefaction wave,
  the corresponding Dafermos function $\mD$ is constant in $J$.
\end{proposition}

\begin{proof}
  By hypothesis,
  $\Uinteg : I \to \Uopen$ is a rarefaction integral curve for $\lambda$ and
  $w : J \to I$ is a homeomorphism
  such that $v := \Uinteg\comp w$ satisfies
  $\lambda\bigl(v(\xi)\bigr) = \xi$ for all $\xi \in J$.

  Define $\widehat{\lambda} := \lambda\comp\Uinteg$,
  $\widehat{G} := \Uinteg$, $\widehat{F} := f\comp\Uinteg$, and
  \begin{equation}
    \widehat{\mD}(\eta) :=
    -\widehat{\lambda}(\eta)\,\widehat{G}(\eta) + \widehat{F}(\eta)
    + \int_{\eta_0}^\eta \widehat{G}(\eta')
    \,\frac{d\widehat{\lambda}}{d\eta}(\eta')\,d\eta'.
    \label{eq:D_hat}
  \end{equation}
  when $\eta_0$, $\eta \in I$.
  These functions are $C^1$ with respect to $\eta \in I$.
  Observe that $\widehat{\mD}(\eta)$ is independent of $\eta$ on $I$ because
  \begin{align}
    \frac{d}{d\eta} \widehat{\mD}
    & = -\frac{d\widehat{\lambda}}{d\eta}\,\widehat{G}
    -\widehat{\lambda}\,\frac{d\widehat{G}}{d\eta} + \frac{d\widehat{F}}{d\eta}
    + \widehat{G}\,\frac{d\widehat{\lambda}}{d\eta}
    = -\widehat{\lambda}\,\frac{d\widehat{G}}{d\eta} +
    \frac{d\widehat{F}}{d\eta} \\
    & = \left(-\lambda\,I + A\right)\comp\Uinteg\,\frac{d\Uinteg}{d\eta} = 0.
  \end{align}

  From $v = \Uinteg\comp w$ we deduce that
  $\widehat{G}\comp w = v$, $\widehat{F}\comp w = f(v)$,
  and $\widehat{\lambda}\bigl(w(\xi)\bigr) \equiv \xi$.
  Also,
  the $C^1$ change of integration variable $\xi' = \widehat{\lambda}(\eta')$
  and the relationship $w = \widehat{\lambda}^{-1}$ imply that
  $\eta' = w(\xi')$
  and $d\xi' = (d\widehat{\lambda}/d\eta)\,d\eta'$,
  so that
  \begin{equation}
    \int_{\eta_0}^\eta \widehat{G}(\eta')
    \,\frac{d\widehat{\lambda}}{d\eta}(\eta')\,d\eta'
    = \int_{\widehat{\lambda}(\eta_0)}^{\widehat{\lambda}(\eta)}
    \widehat{G}\bigl(w(\xi')\bigr)\,d\xi'
    = \int_{\xi_0}^{\widehat{\lambda}(\eta)} v(\xi')\,d\xi'
  \end{equation}
  with $\xi_0 := \widehat{\lambda}(\eta_0)$.
  Setting $\eta = w(\xi)$ in Eq.~\eqref{eq:D_hat} gives
  \begin{equation}
    \widehat{\mD}\bigl(w(\xi)\bigr)
    = -\xi\,v(\xi) + f\bigl(v(\xi)\bigr)
    + \int_{\xi_0}^\xi v(\xi')\,d\xi' = \mD(\xi).
  \end{equation}
  As $\widehat{\mD}$ is constant,
  $\mD$ is constant.
\end{proof}

\subsection{Rarefaction intervals}
\label{sec:rarefaction_intervals}

We now prove the counterpart to Prop.~\ref{prop:constancy}:
on an interval where a solution is continuous and \emph{resonant},
it is a rarefaction wave.

\begin{theorem}
  \label{th:rarefaction_intervals}
  Suppose that every point in a nonempty bounded open interval $J
  \subseteq \Rset$
  is an \ESSIM{} continuity point
  for the self-similar weak solution $v$.
  Let $\overline{v}$ denote the continuous representative of $v$ in $J$.
  If $\overline{v}$ is \emph{resonant} in $J$, meaning
  \begin{equation}
    \text{$\det\left[-\xi\,I + A\bigl(\overline{v}(\xi)\bigr)\right] = 0$
    for all $\xi \in J$,}
    \label{eq:resonance_in_theorem}
  \end{equation}
  then $\overline{v}$ has bounded variation in $J$
  and lies along a unique rarefaction integral curve for $\lambda$.
\end{theorem}

\begin{proof}
  Let $\xi_\star \in J$,
  set $u_\star := \overline{v}(\xi_\star)$,
  and invoke Corollary~\ref{cor:straightening_rarefaction_integral_curves} to
  locally straighten rarefaction integral curves for
  a characteristic speed map $\lambda : \Uopen \to \Rset$,
  using a local diffeomorphism $G : I \times B \to \Uopen$
  such that $u_\star = G(0, 0)$.
  As discussed in Sec.~\ref{sec:locally_transformed_solution},
  we let $V := (w, z)$ with $w \in I \subseteq \Rset$ and $z \in B
  \subseteq \Rset^{n-1}$
  and transform coordinates through $v = G(V)$.
  Define $F := f\comp G$ and $\mu := \lambda\comp G$.
  The crucial features of $(w, z)$-coordinates are that
  \begin{align}
    & \text{$-\mu\,G_w + F_w = 0$ throughout $I \times B$ and}
    \label{eq:partial_w}                                        \\
    & \text{$(-\mu\,G_z + F_z)(0, 0)$ has rank $n - 1$.}
    \label{eq:partial_z}
  \end{align}
  Since $(w, z) \mapsto (-\mu\,G_z + F_z)(w, z)$ is continuous and
  has rank $n-1$ at $(0,0)$,
  after shrinking $I\times B$ we may assume that
  $(-\mu\,G_z + F_z)(w,z)$ has rank $n - 1$ throughout $I \times B$.

  Consider $V_1$, $V_2 \in I \times B$, which is convex,
  and introduce $W(\tau) := (1 - \tau)\,V_1 + \tau\,V_2$ for $\tau \in [0, 1]$.
  Observe that
  \begin{equation}
    \begin{split}
      & \frac{d}{d\tau}\left(-\mu\,G + F\right)\bigl(W(\tau)\bigr)
      = (-\mu\,G_w + F_w)\bigl(W(\tau)\bigr)\cdot (w_2 - w_1)                 \\
      & \qquad\qquad + (-\mu\,G_z + F_z)\bigl(W(\tau)\bigr)\cdot (z_2 - z_1)
      - G\bigl(W(\tau)\bigr)\,\frac{d}{d\tau}\,\mu\bigl(W(\tau)\bigr)
    \end{split}
  \end{equation}
  for all $\tau \in [0, 1]$.
  The key to the proof is that,
  by virtue of Eq.~\eqref{eq:partial_w},
  the first term on the right-hand side disappears.
  Integrating this identity with respect to $\tau$ from $0$ to $1$ yields
  \begin{equation}
    \begin{split}
      & \left(-\mu\,G + F\right)(V_2) - \left(-\mu\,G + F\right)(V_1)
      \\
      & \qquad = \int_0^1 (-\mu\,G_z + F_z)\bigl(W(\tau)\bigr)\,d\tau
      \cdot (z_2 - z_1)
      - \int_0^1 G\bigl(W(\tau)\bigr)
      \,\frac{d}{d\tau}\,\mu\bigl(W(\tau)\bigr)\,d\tau.
    \end{split}
    \label{eq:jump_in_D}
  \end{equation}

  Let $J_\star$ denote the connected component of
  $\overline{v}^{\,-1}[\Uopen]$ that contains $\xi_\star$.
  Define the continuous function
  $V := G^{-1}\comp\left(\overline{v}\restrict_{J_\star}\right)$
  and denote its component functions by $V = \bigl(w, z\bigr)$.
  The resonance hypothesis is that
  $\xi = \mu\bigl(V(\xi)\bigr)$ for all $\xi \in J_\star$.
  In Eq.~\eqref{eq:jump_in_D}
  we substitute $V(\zeta)$ for $V_1$ and $V(\xi)$ for $V_2$.
  By Eq.~\eqref{eq:Dafermos_function_transformed},
  the left-hand side of Eq.~\eqref{eq:jump_in_D}
  equals
  \begin{equation}
    \left(-\xi\,G + F\right)\bigl(V(\xi)\bigr)
    - \left(-\zeta\,G + F\right)\bigl(V(\zeta)\bigr)
    = \mD(\xi) - \mD(\zeta) -\int_{\zeta}^\xi G\bigl(V(\xi')\bigr)\,d\xi'.
    \label{eq:LHS_prelim}
  \end{equation}
  Here $\mD(\xi) - \mD(\zeta) = 0$
  by Eq.~\eqref{eq:D_xi_minus_D_zeta=0_transformed}.
  Because the integrand equals $G\bigl(V(\zeta)\bigr) + o(1)$
  uniformly for $\xi'$ between $\zeta$ and $\xi$,
  \begin{equation}
    \left(-\xi\,G + F\right)\bigl(V(\xi)\bigr)
    - \left(-\zeta\,G + F\right)\bigl(V(\zeta)\bigr)
    = \text{$-\left[G\bigl(V(\zeta)\bigr) + o(1)\right] (\xi - \zeta)$
    as $\xi \to \zeta$.}
    \label{eq:LHS}
  \end{equation}

  On the other hand,
  the first term on the right-hand side of Eq.~\eqref{eq:jump_in_D}
  equals
  \begin{equation}
    \text{$\left[(-\mu\,G_z + F_z)\bigl(V(\zeta)\bigr) + o(1)\right]
      \left[z(\xi) - z(\zeta)\right]$
    as $\xi \to \zeta$}
    \label{eq:RHS_1}
  \end{equation}
  because $(-\mu\,G_z + F_z)\bigl(W(\tau)\bigr)
  = (-\mu\,G_z + F_z)\bigl(V(\zeta)\bigr) + o(1)$
  uniformly for $\tau \in [0, 1]$.
  Finally,
  the second term on the right-hand side of Eq.~\eqref{eq:jump_in_D}
  equals
  \begin{equation}
    \text{$-\left[G\bigl(V(\zeta)\bigr) + o(1)\right] (\xi - \zeta)$
    as $\xi \to \zeta$}
    \label{eq:RHS_2}
  \end{equation}
  because $G\bigl(W(\tau)\bigr) = G\bigl(V(\zeta)\bigr) + o(1)$
  uniformly for $\tau \in [0, 1]$ and
  \begin{equation}
    \int_0^1 \frac{d}{d\tau}\,\mu\bigl(W(\tau)\bigr)\,d\tau
    = \mu\bigl(V(\xi)\bigr) - \mu\bigl(V(\zeta)\bigr)
    = \xi - \zeta.
  \end{equation}

  Consequently,
  Eq.~\eqref{eq:jump_in_D} says that
  \begin{equation}
    \text{$\left[(-\mu\,G_z + F_z)\bigl(V(\zeta)\bigr) + o(1)\right]
    \left[z(\xi) - z(\zeta)\right] = o(\abs{\xi - \zeta})$ as $\xi \to \zeta$.}
  \end{equation}
  Because of property~\eqref{eq:partial_z},
  this equation entails that
  $z(\xi)$ is differentiable,
  and its derivative is $0$, at $\xi = \zeta$.
  As $\zeta \in J_\star$ is arbitrary,
  $z$ has vanishing derivative throughout $J_\star$.
  Also, $\xi_\star \in J_\star$ and $z(\xi_\star) = 0$.
  Thus, $z(\xi) \equiv 0$ in $J_\star$.
  In other words,
  \begin{equation}
    \text{$\overline{v}(\xi) \equiv G\bigl(w(\xi), 0\bigr)$
    for $\xi \in J_\star$.}
    \label{eq:lies_along}
  \end{equation}

  Let $I_\star \subseteq I$ denote the image of $J_\star$ under $w$,
  and let $\mW : I_\star \to I \times B$ with $\omega \mapsto (\omega, 0)$
  be the corresponding rarefaction integral curve for $\mu$,
  which is mapped by $G$ to the rarefaction integral curve for $\lambda$.
  Then Eq.~\eqref{eq:lies_along}
  says that
  \begin{equation}
    \overline{v} = G\comp\mW\comp w
  \end{equation}
  in $J_\star$.
  Let us denote the characteristic speed along $\mW$ by
  $\widehat{\lambda} := \mu\comp\mW = \lambda\comp G\comp \mW$.
  As $\widehat{\lambda}\comp w = \lambda\comp \overline{v}$,
  the resonance hypothesis~\eqref{eq:resonance_in_theorem}
  implies that
  \begin{equation}
    \xi = \widehat{\lambda}\bigl(w(\xi)\bigr)
  \end{equation}
  for all $\xi \in J_\star$.
  Consequently, $w$ is injective, hence a homeomorphism,
  so that $w : J_\star \to I_\star$ reparametrizes $\mW$ by speed $\xi$.
  Thus, $\overline{v}$ lies along $\mW$ in $J_\star$.
  Since $\widehat{\lambda}$ is strictly monotone on $I_\star$
  (by Lemma~\ref{lem:rarefaction_wave})
  and $w=\widehat{\lambda}^{-1}$ on $J_\star$,
  $w$ has bounded variation on $J_\star$.
  Because $\omega \mapsto G(\omega, 0)$ is a $C^1$ map on $I_\star$,
  $\overline{v}(\xi) = G\bigl(w(\xi), 0\bigr)$ has bounded variation
  on $J_\star$.

  We perform this construction for each point $\xi_\star \in J$,
  obtaining the corresponding construction intervals $J_\star \subseteq J$
  in which $\overline{v}$ lies along a rarefaction integral curve.
  Any compact subinterval of $J$
  is covered by these intervals for finitely many points,
  and if two such intervals overlap,
  then $\overline{v}$ lies along the same integral curve in their union
  because the two rarefaction integral curves share a common point
  and the rarefaction integral curve through a point is unique.
  Thus, $\overline{v}$ lies along a single integral curve throughout $J$.
  Because $J$ is bounded,
  finitely many such intervals $J_\star$ cover $J$,
  and the total variation of $\overline{v}$ on $J$
  is bounded by the sum of the variations on this finite cover.
  Therefore, $\overline{v}$ has bounded variation in $J$.
\end{proof}

\section{Discontinuities}
\label{section:discontinuities}

\subsection{Discontinuities}
\label{sec:discontinuities}

Suppose that a self-similar weak solution $v$
of system~\eqref{eq:conslaw}
has an \ESSIM{} discontinuity at $\xi_\star \in \Rset$.
Then the \ESSIM{} accumulation sets
$\Aessim^-(v; \xi_\star)$ and $\Aessim^+(v; \xi_\star)$
are not equal to the same singleton.
For example, they could be different singletons (at a jump discontinuity)
or they could have multiple elements.
Nonetheless,
if $u^- \in \Aessim^-(v; \xi_\star)$
and $u^+ \in \Aessim^+(v; \xi_\star)$,
then $(u^-, \xi_\star, u^+)$ is an R-H jump.

\begin{proposition}
  \label{prop:discontinuities}
  Let $\xi_\star \in \Rset$ be an \ESSIM{} discontinuity point
  for the self-similar weak solution $v$.
  For any $u^- \in \Aessim^-(v; \xi_\star)$
  and $u^+ \in \Aessim^+(v; \xi_\star)$,
  the Rankine--Hugoniot condition
  \begin{equation}
    -\xi_\star\,(u^+ - u^-) + f(u^+) - f(u^-) = 0
    \label{eq:Rankine--Hugoniot_repeated}
  \end{equation}
  is satisfied.
\end{proposition}

\begin{proof}
  Let $\mF$ be the moving frame flux of $v$.
  By Eq.~\eqref{eq:Aessim_union},
  $u^-$, $u^+ \in \Aessim(v; \xi_\star)$,
  so that by Lemma~\ref{lem:discontinuity_lemma},
  $-\xi_\star\,u^- + f(u^-) = \mF(\xi_\star) = -\xi_\star\,u^+ + f(u^+)$.
\end{proof}

\begin{remark}
  \label{rem:proper}
  Although an \ESSIM{} accumulation set for a member of $L^\infty$ is compact,
  it can be uncountable.
  However,
  for a self-similar weak solution $v$,
  Lemma~\ref{lem:discontinuity_lemma}
  entails that $\Aessim(v; \xi_\star)$ is contained in a level set
  \begin{equation}
    \left\{\, u \in \bbU_{SH} \,:\, -\xi_\star\,u + f(u) =
    \mF_\star \,\right\}
  \end{equation}
  for some $\mF_\star \in \Rset^n$.
  Such level sets are severely constrained.
  For example,
  suppose the $C^2$ flux function $f$ is a \emph{proper map}
  (\ie the preimage under $f$ of a compact set is compact).
  If $\mF_\star \in \Rset^n$ is a \emph{regular value} of
  the map $u \mapsto -\xi_\star\,u + f(u)$,
  then the equation $-\xi_\star\,u + f(u) = \mF_\star$
  has finitely many solutions.
  By Sard's theorem, almost every $\mF_\star \in \Rset^n$ is a regular value.
  \qed
\end{remark}

The next result is a strengthened variant of
Prop.~\ref{prop:discontinuities}.
Central to its proof
is identifying left and right \ESSIM{} accumulation states that are distinct.
(In contrast, at a point where $v$ fails to be approximately continuous,
such distinct states might not exist.)
Lemma~\ref{lem:distinct}
is invoked in the proofs of Props.~\ref{prop:constancy_at_infinity}
and~\ref{prop:open}
below.

\begin{lemma}
  \label{lem:distinct}
  Let $\xi_\star \in \Rset$ be an \ESSIM{} discontinuity point
  for the self-similar weak solution $v$.
  There exist $u^- \in \Aessim^-(v; \xi_\star)$
  and $u^+ \in \Aessim^+(v; \xi_\star)$ with $u^+ \ne u^-$,
  \ie $(u^-, \xi_\star, u^+)$ is an R-H jump.
  Moreover,
  the matrix~\eqref{eq:matrix_M}
  is singular:
  \begin{equation}
    \det M(u^-, \xi_\star, u^+) = 0.
    \label{eq:singular}
  \end{equation}
\end{lemma}

\begin{proof}
  By the hypothesis that $v$ has an \ESSIM{} discontinuity at $\xi_\star$,
  it is not the case that
  $\Aessim^-(v; \xi_\star) = \Aessim^+(v; \xi_\star) = \{\ell\}$
  for some $\ell \in \Rset^n$.
  Hence, there exist $u^- \in \Aessim^-(v; \xi_\star)$
  and $u^+ \in \Aessim^+(v; \xi_\star)$
  such that $u^+ \ne u^-$.
  The Rankine--Hugoniot condition~\eqref{eq:Rankine--Hugoniot_repeated},
  written in the form~\eqref{eq:R-H_with_matrix_M},
  says that the nonzero vector $u^+ - u^-$
  belongs to the kernel of $M(u^-, \xi_\star, u^+)$.
\end{proof}

A consequence of
Prop.~\ref{prop:constancy}
and Lemma~\ref{lem:distinct}
is that a self-similar weak solution is essentially constant outside
a bounded interval.

\begin{proposition}
  \label{prop:constancy_at_infinity}
  Suppose $v$ is a self-similar weak solution
  of system~\eqref{eq:conslaw}.
  There exist $\xi_L$, $\xi_R \in \Rset$ such that
  $v$ is essentially constant in $(-\infty, \xi_L)$ and $(\xi_R, \infty)$.
\end{proposition}

\begin{proof}
  By Lemma~\ref{lem:essential_image},
  the essential image of $v$ is a compact subset of $\bbU_{SH}$.
  Accordingly, there exist $\xi_L$, $\xi_R \in \Rset$ such that
  if $\xi > \xi_R$ or $\xi < \xi_L$,
  then
  \begin{align}
    \det \left[-\xi\,I + A(u)\right] & \ne 0, \\
    \det M(u^-, \xi, u^+)            & \ne 0
  \end{align}
  for all $u$, $u^-$, $u^+ \in \essim v$.
  Suppose that $\xi \in (-\infty, \xi_L)$ or $\xi \in (\xi_R, \infty)$.
  By Lemma~\ref{lem:distinct},
  $\xi$ is not an \ESSIM{} discontinuity point for $v$.
  By Prop.~\ref{prop:constancy},
  $v$ is essentially constant in the open intervals
  $(-\infty, \xi_L)$ and $(\xi_R, \infty)$.
\end{proof}

The values of $v$ in these intervals, denoted $u_L$ and $u_R$,
are the data for the Riemann problem solved by $v$.

\subsection{Traveling waves}
\label{sec:traveling_waves}

Proposition~\ref{prop:discontinuities}
has implications for the admissibility of discontinuities.
Suppose the system of conservation laws~\eqref{eq:conslaw}
is intended to model solutions of the hyperbolic-parabolic system
\begin{equation}
  u_t + f(u)_x = \varepsilon\,\left[ B(u)\,u_x \right]_x
  \label{eq:diffusive_conslaw}
\end{equation}
in the limit $\varepsilon \decr 0$.
Fix $\varepsilon > 0$ and specify a propagation speed $\xi_\star \in \Rset$.
A solution of the form
\begin{equation}
  u^\varepsilon(x, t) := w\bigl((x - \xi_\star\,t)/\varepsilon\bigr)
\end{equation}
is called a \emph{traveling wave}.
Introduce the variable $\zeta := (x - \xi_\star\,t)/\varepsilon$.
Then $w$ is a function of $\zeta$ that satisfies
the \emph{traveling wave dynamical system}
\begin{equation}
  B(w)\,\frac{d w}{d\zeta} = -\xi_\star\,w + f(w) - \mF_\star
  \label{eq:dyn_sys}
\end{equation}
for some constant of integration $\mF_\star \in \Rset^n$.

An equilibrium for the ODE~\eqref{eq:dyn_sys}
is a state $u \in \bbU$ that solves
\begin{equation}
  -\xi_\star\,u + f(u) = \mF_\star.
\end{equation}
If $u^- \in \bbU$ and $u^+ \in \bbU$ are equilibria,
then
\begin{equation}
  -\xi_\star\,u^- + f(u^-) = \mF_\star = -\xi_\star\,u^+ + f(u^+),
\end{equation}
so that the Rankine--Hugoniot condition~\eqref{eq:Rankine--Hugoniot}
holds for speed $s = \xi_\star$.
In other words,
$u^-$ and $u^+$ are distinct equilibria for system~\eqref{eq:dyn_sys}
if and only if $(u^-, \xi_\star, u^+)$ is an R-H jump.

Suppose that $w$ is an orbit for system~\eqref{eq:dyn_sys}
that leads from $u^-$ to $u^+$,
\ie $w(\zeta) \to u^-$ as $\zeta \to -\infty$
and $w(\zeta) \to u^+$ as $\zeta \to \infty$.
Now fix $(x, t)$ and let $\varepsilon \decr 0$:
\begin{equation}
  u^\varepsilon(x, t) \to u^0(x, t) :=
  \begin{cases}
    u^- & \text{if $x < \xi_\star\,t$}, \\
    u^+ & \text{if $x > \xi_\star\,t$}.
  \end{cases}
\end{equation}
The limiting function $u^0$ is the self-similar solution
associated with the R-H jump $(u^-, \xi_\star, u^+)$,
which contains a single jump discontinuity.
In this sense,
an isolated jump discontinuity corresponds to
the limit as $\varepsilon \decr 0$
of a traveling wave for system~\eqref{eq:diffusive_conslaw}
if and only if an orbit for system~\eqref{eq:dyn_sys}
leads from $u^-$ to $u^+$.

Evidently,
the traveling wave system~\eqref{eq:dyn_sys}
corresponding to a specified speed $\xi_\star \in \Rset$
and flux vector $\mF_\star \in \Rset^n$
may be expressed in terms of the moving frame flux function $f^{\xi_\star}$:
\begin{equation}
  B(w)\,\frac{d w}{d\zeta} = f^{\xi_\star}(w) - \mF_\star.
  \label{eq:dyn_sys_moving_frame}
\end{equation}
Equilibria for this system are states $u \in \bbU$ that solve
\begin{equation}
  f^{\xi_\star}(u) = \mF_\star.
  \label{eq:equilibrium_equation}
\end{equation}
By Lemma~\ref{lem:discontinuity_lemma}
every state $u \in \Aessim(v; \xi_\star)$,
\ie every \ESSIM{} accumulation value of $v$ at $\xi_\star$,
is an equilibrium.

Suppose that $v$ is a self-similar weak solution.
According to Theorem~\ref{th:Dafermos},
\begin{equation}
  \text{$f^{\xi_\star}\bigl(v(\xi)\bigr) = \mF(\xi)$ for almost every
  $\xi \in \Rset$}
\end{equation}
with $\mF$ being a Lipschitz continuous function.
At each $\xi_\star \in \Rset$,
we obtain a traveling wave ODE
by setting $\mF_\star := \mF(\xi_\star)$
in system~\eqref{eq:dyn_sys_moving_frame}.
From this perspective,
a self-similar weak solution $v$
determines a traveling wave ODE at each $\xi \in \Rset$,
and its \ESSIM{} accumulation values at $\xi$ are equilibria for this system.

\section{Solution structure}
\label{section:solution_structure}

\subsection{Partition of the speed axis}
\label{sec:partition_of_the_speed_axis}

Suppose $v$ is a self-similar weak solution
of system~\eqref{eq:conslaw}.
Closely paralleling Dafermos~\cite{Daf08,Daf26},
we define disjoint sets associated to $v$:
\begin{itemize}
  \item $\mS$ is the set of \ESSIM{} discontinuity points;
  \item $\mC$ is the set of \ESSIM{} continuity points
    that are \emph{non-resonant}, in that
    \begin{equation}
      \det\left[-\xi\,I + A\bigl(\overline{v}(\xi)\bigr)\right] \ne 0;
      \label{eq:not_an_eigenvalue}
    \end{equation}
  \item $\mW$ is the set of \ESSIM{} continuity points
    that are \emph{resonant}, in that
    \begin{equation}
      \det\left[-\xi\,I + A\bigl(\overline{v}(\xi)\bigr)\right] = 0.
      \label{eq:resonant}
    \end{equation}
\end{itemize}
Being disjoint, these sets form a partition of $\Rset$,
the axis for the speed variable $\xi = x/t$:
\begin{equation}
  \Rset = \mS \sqcup \mC \sqcup \mW,
  \label{eq:initial_partition}
\end{equation}

Dafermos~\cite{Daf26} recognized the following crucial topological properties.
The proof relies on Lemma~\ref{lem:distinct}
for \ESSIM{} discontinuities.

\begin{proposition}
  \label{prop:open}
  The set $\mS \sqcup \mW$ is closed and the set $\mC$ is open.
\end{proposition}

\begin{proof}
  We show that the set $\mS \sqcup \mW$ is closed;
  then the complementary set $\mC$ is open.
  Suppose that a sequence in $\mS \sqcup \mW$ converges to $\xi_\star
  \in \Rset$.
  In case $\xi_\star$ is a point of discontinuity for $v$,
  \ie $\xi_\star \in \mS$,
  then the limit belongs to $\mS \sqcup \mW$.
  The other case is that $\xi_\star$ is a point of continuity for $v$.
  By passing to a subsequence, we have two possibilities:
  $\xi_\star$ is the limit of either
  (i)~a sequence in $\mW$ or (ii)~a sequence in $\mS$.

  If $\{\xi_k\}_{k = 1}^\infty$ is a sequence in $\mW$ converging to
  $\xi_\star$,
  then $\det\left[-\xi\,I + A\bigl(\overline{v}(\xi_k)\bigr)\right] = 0$
  for all $k \ge 1$.
  As $\overline{v}$ is continuous at $\xi_\star$,
  $\overline{v}(\xi_k) \to \overline{v}(\xi_\star)$ as $k \to \infty$.
  Therefore,
  $\det\left[-\xi\,I + A\bigl(\overline{v}(\xi_\star)\bigr)\right] = 0$,
  \ie $\xi_\star \in \mW$.

  Suppose instead $\{\xi_k\}_{k = 1}^\infty$ is a sequence in $\mS$
  that converges to $\xi_\star$.
  By Lemma~\ref{lem:distinct}
  there exist distinct states $u_k^\pm \in \Aessim^\pm(v; \xi_k)$
  such that $\det M(u_k^-, \xi_k, u_k^+) = 0$.
  As $\xi_\star$ is a point of continuity,
  $u_k^-$ and $u_k^+$ tend to the same limit as $k \to \infty$,
  \viz $\overline{v}(\xi_\star)$,
  so $\det M(\overline{v}(\xi_\star), \xi_\star,
  \overline{v}(\xi_\star)) = 0$.
  By the definition~\eqref{eq:matrix_M}
  of $M$, $\xi_\star \in \mW$.
\end{proof}

We now characterize the nature of $v$
within each of the sets $\mC$, $\mS$, and $\mW$.

\subsection{Constant states}
\label{ref:C}
Being open by Prop.~\ref{prop:open},
$\mC$ is the disjoint union of its connected components.
Because we may label each connected component by
a single rational number chosen within it,
there are countably many connected components of $\mC$.
Let $J$ be such a connected component.

By definition of $\mC$,
all points in $J$ are \ESSIM{} continuity points of $v$.
Let $\overline{v}$ be the continuous representative of $v$ in $J$.
Also by definition of $\mC$,
$\overline{v}$ is non-resonant in $J$,
\ie for every $\xi \in J$,
$\xi$ is not an eigenvalue of $A\bigl(\overline{v}(\xi)\bigr)$.
By Prop.~\ref{prop:constancy},
$\overline{v}$ is constant throughout $J$.
Thus, every point in $\mC$ belongs to an open interval
in which the solution $\overline{v}$ is constant.
In other words,
$\overline{v}$ is \emph{locally constant} in $\mC$.

\begin{lemma}
  \label{lem:endpoints}
  Let $J := (\xi_a, \xi_b)$ be a nonempty bounded open interval.
  If $J \subseteq \mC$ or $J \subseteq \mW$,
  then the limits $v_a := \lim_{\xi\decr\xi_a} \overline{v}(\xi)$
  and $v_b := \lim_{\xi\incr\xi_b} \overline{v}(\xi)$ exist.
  Moreover,
  \begin{equation}
    \text{$\Aessim^+(v; \xi_a) = \{ v_a \}$
    and $\Aessim^-(v; \xi_b) = \{ v_b \}$.}
    \label{eq:singletons}
  \end{equation}
\end{lemma}

We refer to $v_a$ and $v_b$ as then \emph{endpoint states of $v$ in $J$}.

\begin{proof}
  Since $J \subseteq \mC$ or $J \subseteq \mW$,
  every point of $J$ is a point of \ESSIM{} continuity for $v$.
  Hence, by Lemma~\ref{lem:continuity_lemma_state_function},
  there exists a continuous function $\overline{v} : J \to \Rset^n$
  such that $\overline{v} = v$ a.e.\ in $J$.
  If $J \subseteq \mC$, then $\overline{v}$ is constant on $J$
  by Prop.~\ref{prop:constancy};
  if $J \subseteq \mW$, then $\overline{v}$ is BV on $J$
  by Theorem~\ref{th:rarefaction_intervals}.
  Therefore, the indicated one-sided limits $v_a$ and $v_b$ exist.
  We prove that $\Aessim^+(v;\xi_a)=\{v_a\}$;
  the proof that $\Aessim^-(v;\xi_b)=\{v_b\}$ is analogous.

  Let $U$ be a neighborhood of $v_a$.
  By the definition of $v_a$, there exists $\delta > 0$ such that
  $\overline{v}(\xi) \in U$ for all $\xi \in (\xi_a, \xi_a + \delta)$.
  As $\overline{v} = v$ a.e.\ in $J$,
  \begin{equation}
    \abs{v^{-1}[U] \cap (\xi_a, \xi_a + r)} > 0
  \end{equation}
  for every $r \in (0,\delta)$.
  Hence, $v_a \in \Aessim^+(v;\xi_a)$.

  Now let $z \neq v_a$.
  Choose disjoint neighborhoods $V$ of $z$ and $U$ of $v_a$.
  Shrinking $\delta$ if necessary, we have
  $\overline{v}(\xi) \in U$ for all $\xi \in (\xi_a, \xi_a + \delta)$.
  Since $U \cap V = \emptyset$ and $\overline{v} = v$ a.e.\ in $J$,
  \begin{equation}
    \abs{v^{-1}[V] \cap (\xi_a, \xi_a + \delta)} = 0.
  \end{equation}
  Therefore, $z \notin \Aessim^+(v;\xi_a)$.
\end{proof}

If $v$ is a reduced state function,
then $v \in L^\infty(\Rset; \Rset^n)$
and $\essim v \subseteq \bbU_{SH}$,
so that there is a compact subset of $\bbU_{SH}$
that contains $v(\xi)$ for almost every $\xi \in \Rset$.
As a result, there exists $\Lambda > 0$ such that
\begin{equation}
  \lambda_{k + 1}(u) - \lambda_k(u) \ge \Lambda
  \label{eq:lower_bound}
\end{equation}
for all families $k = 1$, $\ldots$, $n - 1$ and all $u$ in this compact subset.

\begin{lemma}
  \label{lem:lower_bound}
  Let $J$ be a connected component of $\mC$.
  If both endpoints of $J$ belong to $\mW$,
  then $|J| \ge \Lambda$.
\end{lemma}

\begin{proof}
  Write $J := (\xi_a, \xi_b)$
  and let $\overline{v}(\xi) \equiv u_J$ for $\xi \in J$.
  If $\xi_a$ and $\xi_b$ belong to $\mW$,
  then by Lemma~\ref{lem:endpoints},
  $\xi_a = \lambda_j(u_J)$ and $\xi_b = \lambda_k(u_J)$
  for some families $j$ and $k$.
  Necessarily, $j < k$.
  (In fact, $k = j + 1$.)
  The lower bound~\eqref{eq:lower_bound}
  implies that $|J| = \xi_b - \xi_a \ge \Lambda$.
\end{proof}

\begin{remark}
  \label{rem:lower_bound}
  If either endpoint of a connected component $J$ of $\mC$ belongs to $\mS$,
  then $|J|$ can be arbitrarily small.
  \qed
\end{remark}

\begin{proposition}
  \label{prop:local_finiteness}
  If $K := (\xi_a, \xi_b)$ is a nonempty bounded open interval
  contained in $\mC \sqcup \mW$,
  then $\mC \cap K$ and $\mW \cap K$ have finitely many connected components.
  Moreover, $\xi_a$ is the left endpoint of a nonempty bounded open
  interval $J_a$
  contained in $\mC$ or $\mW$,
  and similarly for $\xi_b$.
\end{proposition}

\begin{proof}
  At most one connected component of $\mC \cap K$
  has left endpoint $\xi_a$,
  and at most one has right endpoint $\xi_b$.
  Let $J$ be any other connected component $\mC \cap K$.
  Then $J$ has endpoints in $K \subseteq \mC \sqcup \mW$,
  but because a connected component is maximal,
  neither endpoint belongs to $\mC$.
  Therefore, $J$ is a connected component of $\mC$
  with both endpoints belonging to $\mW$.
  Lemma~\ref{lem:lower_bound}
  says that $|J| \ge \Lambda$.
  Since the interval $K$ has finite length,
  there can be only finitely many
  connected components of $\mC \cap K$.
  Consequently,
  its complement in $K$,
  namely $K \setminus \mC = \mW \cap K$,
  has finitely many connected components.
  As a result, there exists $\zeta_a > \xi_a$
  such that $J_a := (\xi_a, \zeta_a)$
  is contained in $\mC$ or $\mW$,
  and similarly on the right side.
\end{proof}

\subsection{Discontinuities}
\label{sec:S}
Suppose that $\xi_\star \in \mS$.
By Lemma~\ref{lem:distinct},
there exist $u^- \in \Aessim^-(v; \xi_\star)$
and $u^+ \in \Aessim^+(v; \xi_\star)$ such that $u^+ \ne u^-$,
so that $(u^-, \xi_\star, u^+)$ is an R-H jump.
Each accumulation set can have more than one element,
but by Prop.~\ref{prop:discontinuities},
$(u^-, \xi_\star, u^+)$ is an R-H jump for every
$u^- \in \Aessim^-(v; \xi_\star)$ and $u^+ \in \Aessim^+(v; \xi_\star)$.

The traditional assumption is that $v$ has a jump discontinuity at
$\xi_\star$,
\ie $v$ has left and right limits $u^-$ and $u^+$ at $\xi_\star$ with
$u^+ \ne u^-$,
or equivalently,
$\Aessim^-(v; \xi_\star)$ and $\Aessim^+(v; \xi_\star)$ are
distinct singletons.
(See Def.~\ref{def:essential-image_jump_discontinuity}.)
A sufficient condition for a discontinuity to be a jump
is that $\xi_\star$ is an isolated point of $\mS$,
as we now show.

Let $\mSprime$ denote the \emph{derived set of $\mS$},
\ie the set of limit points of $\mS$.
Note that $\mSprime$ might contain points that are not in $\mS$.
The \emph{set of isolated points of $\mS$}
is the set difference of $\mS$ and $\mSprime$:
\begin{equation}
  \iso(\mS) := \mS \setminus \mSprime.
  \label{eq:iso_S}
\end{equation}

\begin{proposition}
  \label{prop:jumps}
  A point $\xi_\star \in \iso(\mS)$ is an \ESSIM{} jump discontinuity point.
\end{proposition}

\begin{proof}
  Because $\xi_\star$ is not a limit point of $\mS$,
  there exists $\delta > 0$ such that $K := (\xi_\star - \delta, \xi_\star)$
  is disjoint from $\mS$,
  so $K \subseteq \mC \sqcup \mW$.
  By Prop.~\ref{prop:local_finiteness},
  $\xi_\star$ is the right endpoint of a nonempty bounded open interval $J$
  contained in $\mC$ or $\mW$.
  By Lemma~\ref{lem:endpoints},
  $\Aessim^-(v; \xi_\star)$ is a singleton.
  Similarly, $\Aessim^+(v; \xi_\star)$ is a singleton,
  so that $v$ has a jump discontinuity at $\xi_\star$.
\end{proof}

\begin{remark}
  \label{eq:discontinuity_limit_point}
  An example of a point in $\mSprime$ is provided by
  the scalar conservation law with the $C^2$ flux function
  \begin{equation}
    f(u) := u^2 + u^5\,\sin(1/u).
  \end{equation}
  The solution of the Riemann problem with $u_L = -1$ and $u_R = 1$
  contains a countably infinite number of shock waves with limit point $u = 0$
  belonging to $\mW$.
  A variant of this example has a limit point in $\mS$.
  \qed
\end{remark}

\subsection{Continuous Waves}
\label{ref:W}
Let $J$ be a nonempty bounded open interval contained in $\mW$.
All points $\xi \in J$ are \ESSIM{} continuity points of $v$.
Let $\overline{v}$ be the continuous representative of $v$ in $J$.
By definition of $\mW$,
$\overline{v}$ is resonant in $J$,
\ie $\xi$ is an eigenvalue of $A\bigl(\overline{v}(\xi)\bigr)$
for all $\xi \in J$.
Theorem~\ref{th:rarefaction_intervals}
entails that $\overline{v}$ is a rarefaction wave within $J$.

Define the \emph{rarefaction set $\mR$ of the solution $v$} to be
\begin{equation}
  \mR := \interior \mW.
  \label{eq:R}
\end{equation}
Being open, this set is the disjoint union
of its countably many connected components,
and the solution is a rarefaction wave in each.
Certain points in the complementary subset $\mW \setminus \mR$
have clear interpretations.
Define
\begin{equation}
  \mE := (\mW \setminus \mR) \setminus \mSprime.
  \label{eq:E}
\end{equation}
We call a point $\xi_\star \in \mE$ a \emph{division point for $v$}.

\begin{proposition}
  \label{prop:E}
  Let $\xi_\star \in \mE$ be a division point.
  Then there exist nonempty bounded open intervals $J_\ell$ and $J_r$,
  each contained in $\mC$ or $\mW$,
  such that
  \begin{itemize}
    \item[(a)] $\xi_\star$ is the right endpoint of $J_\ell$ and
      the left endpoint of $J_r$ and
    \item[(b)] $v$ is continuous in $J_\ell \sqcup \{ \xi_\star \}
      \sqcup J_r$.
  \end{itemize}
  Moreover, either:
  \begin{enumerate}
    \item[(1)] $J_\ell \subseteq \mR$ and $J_r \subseteq \mC$ or
      vice versa; or
    \item[(2)] $J_\ell \subseteq \mC$ and $J_r \subseteq \mC$.
  \end{enumerate}
\end{proposition}

\begin{proof}
  As $\xi_\star$ is not a limit point of $\mS$,
  there exists $\delta > 0$ such that $K_\ell := (\xi_\star -
  \delta, \xi_\star)$
  is disjoint from $\mS$,
  so $K_\ell \subseteq \mC \sqcup \mW$.
  By Prop.~\ref{prop:local_finiteness},
  $\xi_\star$ is the right endpoint of a nonempty bounded open
  interval $J_\ell$ contained in $\mC$ or $\mR$.
  Similarly,
  $\xi_\star$ is the left endpoint of a nonempty bounded open
  interval $J_r$ contained in $\mC$ or $\mR$.
  If $J_\ell$ and $J_r$ were both contained in $\mR$,
  then $\xi_\star$ would belong to $\mR$,
  which is excluded.
  The remaining possibilities are cases~(1) and~(2).
\end{proof}

In case~(1), the division point $\xi_\star \in \mE$
is an endpoint of $J \subseteq \mC$,
where $J = J_\ell$ or $J = J_r$,
and $v$ is constant in $J \sqcup \{ \xi_\star \}$.
In case~(2),
$v$ is constant in $J_\ell \sqcup \{ \xi_\star \} \sqcup J_r$.
Therefore, $\mC \cup \mE$ consists of countably many
(not necessarily open) intervals in which $v$ is constant.
As each point of $\mE$ is an endpoint of a connected component of $\mC$,
which are countable,
$\mE$ is countable.

\begin{remark}
  To illustrate case~(2),
  suppose that the system of conservation laws has $n = 2$ components.
  Consider the constant solution $v(\xi) \equiv u_\star$,
  which solves the trivial Riemann problem with data $u_L = u_R =: u_\star$.
  Denote $\lambda_1^\star := \lambda_1(u_\star)$
  and $\lambda_2^\star := \lambda_2(u_\star)$.
  These speeds are resonant
  because $u_\star = \overline{v}(\lambda_1^\star)$
  and $u_\star = \overline{v}(\lambda_2^\star)$.
  Therefore, they belong to $\mE$,
  not to $\mC$.
  \qed
\end{remark}

\subsection{Refined partition}
\label{ref:refined}

\begin{theorem}
  \label{th:partition}
  Associated to each self-similar weak solution $v$ is the partition
  \begin{equation}
    \Rset = \iso(\mS) \sqcup \mR \sqcup (\mC \sqcup \mE) \sqcup \mSprime.
    \label{eq:refined_partition}
  \end{equation}
  of the speed axis into
  \begin{itemize}
    \item a countable set $\iso(\mS)$ of isolated \ESSIM{}
      discontinuities of $v$;
    \item a countable union of connected components of $\mR$
      in which $v$ is a rarefaction wave;
    \item a countable union of connected components of $\mC$
      in which $v$ is constant;
    \item a countable set $\mE$ of isolated endpoints of continuity
      for $v$; and
    \item the set of limit points $\mSprime$ of \ESSIM{}
      discontinuities of $v$.
  \end{itemize}
\end{theorem}

\begin{proof}
  We partition $\Rset$ into the derived set $\mSprime$ and its complement.
  By definition~\eqref{eq:R},
  \begin{equation}
    \Rset = \mC \sqcup \mS \sqcup \mW
    = \mC \sqcup \mS \sqcup \mR \sqcup (\mW \setminus \mR).
  \end{equation}
  As $\mC$ is open, $\mC \setminus \mSprime = \mC$;
  similarly, $\mR \setminus \mSprime = \mR$.
  By definition~\eqref{eq:iso_S},
  $\mS \setminus \mSprime = \iso(\mS)$,
  which is countable.
  Definition~\eqref{eq:E}
  says that $(\mW \setminus \mR) \setminus \mSprime = \mE$,
  which is also countable.
  Therefore,
  \begin{equation}
    \bigl[\mC \sqcup \mS \sqcup \mR \sqcup (\mW \setminus \mR)\bigr]
    \setminus \mSprime
    = \mC \sqcup \iso(\mS) \sqcup \mR \sqcup \mE.
  \end{equation}
  Because the complement of this set is $\mSprime$,
  we obtain the partition~\eqref{eq:refined_partition}.
\end{proof}

\section{Conclusion}
\label{section:conclusion}

As illustrated in Remark~\ref{eq:discontinuity_limit_point},
a self-similar weak solution
might contain a countably infinite number of \ESSIM{} discontinuities.
As $\mS$ is contained in a compact set $[\xi_L, \xi_R]$
by Prop.~\ref{prop:constancy_at_infinity},
$\mS$ is infinite if and only if $\mSprime$ is nonempty.
Permitting infinitely many discontinuities,
even ones required to satisfy an admissibility condition,
opens the door to anomalous phenomena.

For instance, in Ref.~\cite{PloSchMar26b}
we construct a self-similar weak solution $v$
containing an \ESSIM{} discontinuity at speed $\xi_\star$
that is not a jump discontinuity.
In this example:
\begin{itemize}
  \item $\Aessim^-(v; \xi_\star)$ has three distinct elements;
  \item $v$ also contains an infinite sequence of jump discontinuities
    with speeds $\{\xi_k\}_{k=0}^\infty$ that increase to $\xi_\star$; and
  \item each jump discontinuity is admissible according to a traveling wave criterion,
    each being an under-compressive shock wave,
    \ie its profile is an orbit connecting two saddle points.
\end{itemize}

However, in the common and practical situation where
a Riemann solution has finitely many \ESSIM{} discontinuities,
it consists of finitely many rarefaction waves, R-H jumps, and
constant states.

\begin{theorem}
  \label{th:finite}
  Let $v$ be a self-similar weak solution of system~\eqref{eq:conslaw}.
  Suppose that $v$ has finitely many discontinuities.
  Then $v$ consists of a finite number of rarefaction waves, R-H jumps,
  and constant states.
\end{theorem}

\begin{proof}
  By Prop.~\ref{prop:constancy_at_infinity},
  there exist $\xi_L$, $\xi_R \in \Rset$ such that
  $(-\infty, \xi_L)$ and $(\xi_R, \infty)$ are contained in $\mC$;
  let $L := (\xi_L - \epsilon, \xi_R + \epsilon)$ for some $\epsilon > 0$.
  As $\mS$ is finite,
  $L \setminus \mS$ has finitely many connected components,
  each contained in $\mC \sqcup \mW$.
  If $K$ is one of them,
  then according to Prop.~\ref{prop:local_finiteness},
  $\mC \cap K$ and $\mW \cap K$, have finitely many components,
  which implies that $\mC$ and $\mW$ do too.
  In particular, $\mE$, as characterized
  in Prop.~\ref{prop:local_finiteness}, is finite.
  Therefore,
  $\Rset = \iso(\mS) \sqcup \mR \sqcup (\mC \sqcup \mE)$,
  where $\iso(\mS)$ is the finite set of jump discontinuities of $v$,
  $\mR$ is the union of finitely many bounded open intervals
  in which $v$ is a rarefaction wave,
  and $\mC \sqcup \mE$ is the union of finitely many intervals
  (open, half-open, or closed)
  in which $v$ is constant.
\end{proof}

\section*{Declaration of generative AI use}

During the preparation of this manuscript,
the authors used ChatGPT,
a generative AI tool developed by OpenAI,
for assistance with language editing
and critical review of the mathematical exposition.
The authors reviewed and revised all AI-assisted output
and take full responsibility for the content of this manuscript.

\appendix

\section{Essential image}
\label{sec:essential_image}

In this appendix,
we recall the definition of the essential image of a measurable map
and use it to develop a notion of accumulation set that respects equality a.e.

Set $Y := \Rset^p$.
For a Lebesgue measurable subset $\Yopen \subseteq Y$,
let $\abs{\Yopen}$ denote its measure.
For $y \in Y$ and $r > 0$,
let $B_r(y) \subseteq Y$ denote the open ball with radius $r$ centered at $y$.

Set $Z := \Rset^n$.
Let $\mB$ be the set
of open balls with rational radii and centers with rational coordinates,
which is a countable base for the topology on $Z$.

Let $h : Y \to Z$ be a measurable map:
if $\Zopen \subseteq Z$ is open,
$h^{-1}[\Zopen]$ is measurable.
A measure-theoretic analog of the image $\im h$ of $h$
is defined as follows.

\begin{definition}
  \label{def:essential_image}
  The \emph{essential image} (or \emph{essential range}) \emph{of $h$},
  which we denote $\essim h$,
  is the set of $z \in Z$ such that
  the preimage under $h$ of each neighborhood of $z$ has positive measure.
\end{definition}
\par\noindent
Thus,
$z \in \essim h$ if and only if $\abs{h^{-1}[\Zopen]} > 0$
for any neighborhood $\Zopen$ of $z$.
Equivalently,
$z \notin \essim h$ if and only if
there exists a neighborhood $\Zopen$ of $z$ such that
$\abs{h^{-1}[\Zopen]} = 0$.
Notice that if $h = k$ a.e., $\essim h = \essim k$,
\ie essential images are invariant under equality a.e.

The image of a subset $\Yopen \subseteq Y$ under $h$,
which is the image of the restriction $h\restrict_{\Yopen}$ of $h$ to $\Yopen$,
likewise has a measure-theoretic analog.

\begin{definition}
  \label{def:essential_image_of_subset}
  Let $\Yopen \subseteq Y$ be measurable.
  The \emph{essential image of $\Yopen$ under $h$}
  is the essential image of $h\restrict_{\Yopen}$.
\end{definition}
\par\noindent
In other words,
$z \in \essim h\restrict_{\Yopen}$ if and only if
$\abs{h^{-1}[\Zopen] \cap \Yopen} > 0$
for any neighborhood $\Zopen$ of $z$.

\begin{lemma}
  \label{lem:union}
  Suppose that $\Yopen$ and $\Yopen'$ are measurable.
  Then
  \begin{equation}
    \essim h\restrict_{\Yopen\,\cup\,\Yopen'}
    = \left(\essim h\restrict_{\Yopen}\right)
    \cup \left(\essim h\restrict_{\Yopen'}\right).
  \end{equation}
\end{lemma}

\begin{proof}
  Let $A := h^{-1}[\Zopen] \cap \Yopen$,
  $A' := h^{-1}[\Zopen] \cap \Yopen'$,
  and $B := h^{-1}[\Zopen] \cap \left(\Yopen \cup \Yopen'\right)$.
  As $B = A \cup A'$,
  \begin{equation}
    \max\{\abs{A}, \abs{A'}\} \le \abs{B} \le \abs{A} + \abs{A'},
  \end{equation}
  so that $\abs{B} > 0$ if and only if $\abs{A} > 0$ or $\abs{A'} > 0$.
  Hence, $z \in \essim h\restrict_{\Yopen \cup \Yopen'}$
  if and only if $z \in \essim h\restrict_{\Yopen}$
  or $z \in \essim h\restrict_{\Yopen'}$.
\end{proof}

\begin{lemma}
  \label{lem:monotonicity}
  Suppose that $\Yopen$, $\Yopen'$, $\Nullset \subseteq Y$ are measurable,
  $\Nullset$ has measure zero,
  and $\Yopen \subseteq \Yopen' \cup \Nullset$.
  Then $\essim h\restrict_{\Yopen} \subseteq \essim h\restrict_{\Yopen'}$.
\end{lemma}

\begin{proof}
  Let $z \in \essim h\restrict_{\Yopen}$.
  For any neighborhood $\Zopen$ of $z$,
  $\abs{h^{-1}[\Zopen] \cap \Yopen} > 0$.
  The assumption $\Yopen \subseteq \Yopen' \cup \Nullset$
  implies $\abs{h^{-1}[\Zopen] \cap (\Yopen' \cup \Nullset)} > 0$.
  As $\Nullset$ has measure zero,
  $\abs{h^{-1}[\Zopen] \cap \Yopen'} > 0$.
  Thus, $z \in \essim h\restrict_{\Yopen'}$.
\end{proof}

\begin{lemma}
  \label{lem:essential_image}
  Suppose that $\Yopen \subseteq Y$ has positive measure
  and $h\restrict_{\Yopen}$ is essentially bounded.
  Then $\essim h\restrict_{\Yopen}$ is nonempty and compact,
  and $h(y) \in \essim h\restrict_{\Yopen}$ for almost every $y \in \Yopen$.
\end{lemma}

\begin{proof}
  Denote $W := \essim h\restrict_{\Yopen}$.

  First, we show that $W$ is bounded.
  Let $K > \norm{h\restrict_{\Yopen}}_\infty$.
  If $\overline{B} \subseteq Z$
  denotes the closed ball centered at $0$ with radius $K$,
  then $\overline{B}^{\,c}$ is open and
  \begin{equation}
    h^{-1}\left[\,\overline{B}^{\,c}\,\right] \cap \Yopen
    = \{\, y \in \Yopen \,:\, \norm{h(y)} > K \,\}
  \end{equation}
  has measure zero,
  so that any $z \in \overline{B}^{\,c}$ belongs to $W^c$.
  Thus, $W \subseteq \overline{B}$.

  Let $\mB$ be the countable base for the topology on $Z$,
  and define $\mB_h$ as the set of $B \in \mB$
  such that $\abs{h^{-1}[B] \cap \Yopen} = 0$.
  We show that the complement of $W$ is
  \begin{equation}
    W^c = \bigcup\,\{\, B \,:\, B \in \mB_h \,\}
    \label{eq:union}
  \end{equation}
  as follows:
  \begin{itemize}
    \item If $z \in W^c$,
      then $z$ has a neighborhood $\Zopen$
      such that $\abs{h^{-1}[\Zopen] \cap \Yopen} = 0$.
      Choose $B \in \mB$ such that $z \in B \subseteq \Zopen$.
      As $\abs{h^{-1}[B] \cap \Yopen} \le \abs{h^{-1}[\Zopen] \cap \Yopen} = 0$,
      $B \in \mB_h$.
      Thus, $z \in B$ for some $B \in \mB_h$.
    \item If $z \in B$ for some $B \in \mB_h$,
      then $\abs{h^{-1}[B] \cap \Yopen} = 0$,
      so that $z \in W^c$.
  \end{itemize}

  One consequence is that $W^c$ is the union of open sets, \ie $W$ is closed.
  Being bounded, $W$ is compact.
  In addition,
  Eq.~\eqref{eq:union}
  implies that
  \begin{equation}
    h^{-1}[W^c] \cap \Yopen = \bigcup\,\{\, h^{-1}[B] \cap \Yopen
    \,:\, B \in \mB_h \,\}
  \end{equation}
  is a countable union of sets with measure zero; hence it has measure zero.
  Thus, the set of $y \in \Yopen$ such that $h(y) \notin W$ has measure zero.
  Finally, $W$ is nonempty, for otherwise
  $\Yopen = \left(h^{-1}[W] \cup h^{-1}[W^c]\right) \cap \Yopen$
  would have measure zero,
  contrary to hypothesis.
\end{proof}

The accumulation set (or cluster set) of $h$ at $y \in \Yopen$
is the set of limits of $h(y_k)$ as $k \to \infty$
for sequences $\{y_k\}_{k = 1}^\infty$ in $\Yopen$ that converge to $y$.
We now define an analog based on essential images.

\begin{definition}
  \label{def:essential_accumulation_sets}
  Let $y \in Y$.
  The \emph{\ESSIM{} accumulation set of $h$ at $y$} is
  \begin{equation}
    \Aessim(h; y) := \bigcap_{r > 0} \essim h\restrict_{B_r(y)}.
  \end{equation}
  Elements of this set are the \emph{\ESSIM{} accumulation values of
  $h$ at $y$}.
\end{definition}

\begin{remark}
  \label{rem:naming}
  To avoid confusion with
  a similar, but distinct, concept named an \emph{essential cluster set}
  that appears in the mathematical literature,
  we do not use the simpler term \emph{essential accumulation set}.
\end{remark}

\begin{remark}
  \label{rem:scalar}
  For the scalar case $n = 1$,
  the closely related notions of
  \emph{essential continuity} and \emph{essential limit sets}
  have received a comprehensive development in Ref.~\cite{FelWag98}.
  Even though the definition of essential limit set $E_v(\xi)$
  is based on the total order in $\Rset$,
  whereas the definition of \ESSIM{} accumulation set $\Aessim(v; \xi)$
  is based on essential images,
  these sets coincide.
  In particular, the associated notion of continuity is the same.
  \qed
\end{remark}

\begin{lemma}
  \label{lem:inclusion}
  Suppose that $\Yopen \subseteq Y$ is open
  and $h\restrict_{\Yopen}$ is essentially bounded.
  For any $y \in \Yopen$, $\Aessim(h; y)$ is nonempty and compact.
  Moreover, $h(y) \in \Aessim(h; y)$ for almost every $y \in \Yopen$.
\end{lemma}

\begin{proof}
  Fix $y \in \Yopen$.
  Because $\Yopen$ is open,
  there exists $R > 0$ such that $B_R(y) \subseteq \Yopen$.
  By Lemma~\ref{lem:monotonicity},
  the sets $E_r := \essim h\restrict_{B_r(y)}$ for $r \in (0, R)$
  form a nested family of nonempty, compact subsets of $Y$
  that shrink as $r \decr 0$.
  The intersection $\Aessim(h; y)$ of these sets is compact
  because it is closed and contained in $E_R$.
  It is also nonempty, for otherwise the open sets $E_R \setminus E_r$
  would cover $E_R$; by compactness, finitely many would suffice,
  so that $E_r$ would be empty for some $r > 0$, a contradiction.

  Let $\mB$ be the countable base for the topology on $Z$.
  By the Lebesgue density theorem~\cite[Theorem~1.35]{EvaGar15},
  for any $B \in \mB$ there exists a set $\Nullset_B$ of measure zero such that
  every $y \in h^{-1}[B]\,\setminus\,\Nullset_B$ is a point of density~1,
  meaning that
  \begin{equation}
    \lim_{r \decr 0} \frac{\abs{h^{-1}[B] \cap B_r(y)}}{\abs{B_r(y)}} = 1.
    \label{eq:Lebesgue_density}
  \end{equation}
  Form the union $\Nullset$ of $\Nullset_B$ for $B \in \mB$;
  then $\Nullset$ has measure zero because $\mB$ is countable.

  Now fix $y \in \Yopen\,\setminus\,\Nullset$.
  We show that $h(y) \in \Aessim(h; y)$.
  Again let $R > 0$ be such that $B_R(y) \subseteq \Yopen$
  and let $r \in (0, R)$.
  For any neighborhood $\Zopen$ of $h(y)$,
  choose $B \in \mB$ such that $h(y) \in B \subseteq \Zopen$.
  Then $y\in h^{-1}[B]\,\setminus\,\Nullset_B$
  and Eq.~\eqref{eq:Lebesgue_density}
  ensures that $\abs{h^{-1}[B] \cap B_\delta(y)}/\abs{B_\delta(y)} \ge \half$
  for some $\delta \in (0, r)$.
  Therefore,
  \begin{equation}
    \abs{h^{-1}[\Zopen] \cap B_r(y)} \ge \abs{h^{-1}[B] \cap B_\delta(y)}
    \ge \half\,\abs{B_\delta(y)} > 0.
  \end{equation}
  Thus, $h(y) \in \essim h\restrict_{B_r(y)}$.
  As $r \in (0, R)$ is arbitrary,
  $h(y) \in \Aessim(h; y)$.
\end{proof}

Just as $y$ is a continuity point for $h$
if and only if its accumulation set at $y$ is a singleton,
we make the following definition.

\begin{definition}
  \label{def:essential-image_continuity}
  We say that $y \in Y$ is an \emph{\ESSIM{} continuity point} for $h$
  provided $\Aessim(h; y)$ contains a single element,
  which we denote by $\overline{h}(y)$:
  \begin{equation}
    \Aessim(h; y) = \bigl\{\,\overline{h}(y)\bigr\}.
    \label{eq:essential-image_continuity_point}
  \end{equation}
  Otherwise, we say $y$ is an \emph{\ESSIM{} discontinuity point} for $h$.
\end{definition}

\begin{remark}
  Ess-im continuity is defined in terms of accumulation sets;
  it is not induced by a topology on $Y$.
\end{remark}

\begin{lemma}
  \label{lem:continuity_lemma}
  Suppose that $\Yopen \subseteq Y$ is open
  and $h\restrict_{\Yopen}$ is essentially bounded.
  Assume that every point in $\Yopen$ is an \ESSIM{} continuity point for $h$.
  Then $\overline{h}$ is continuous in $\Yopen$
  and $h = \overline{h}$ a.e.\ in $\Yopen$.
\end{lemma}

\begin{proof}
  Let $y_\star \in \Yopen$
  and let $V \subseteq Z$ be a neighborhood of $z_\star :=
  \overline{h}(y_\star)$.
  As
  \begin{equation}
    \{ z_\star \} = \Aessim(h; y_\star) = \bigcap_{r > 0} \essim
    h\restrict_{B_r(y_\star)},
  \end{equation}
  and the sets $\essim h\restrict_{B_r(y_\star)}$
  form a nested family of nonempty compact sets,
  there exists $\delta > 0$ such that $\essim
  h\restrict_{B_\delta(y_\star)} \subseteq V$.
  Let $B := B_{\delta/2}(y_\star)$; we demonstrate that
  $\im \overline{h}\restrict_B \subseteq V$.
  If $y \in B$, then $B_{\delta/2}(y) \subseteq B_{\delta}(y_\star)$.
  By Lemma~\ref{lem:monotonicity}
  (with $\Nullset = \emptyset$),
  \begin{equation}
    \bigl\{\,\overline{h}(y)\bigr\} \subseteq \essim
    h\restrict_{B_{\delta/2}(y)}
    \subseteq \essim h\restrict_{B_{\delta}(y_\star)} \subseteq V.
  \end{equation}
  Thus, $\overline{h}$ is continuous at $y_\star$.
  By Lemma~\ref{lem:inclusion},
  $h(y) \in \Aessim(h; y) = \bigl\{\,\overline{h}(y)\bigr\}$
  for almost every $y \in \Yopen$.
  Therefore, $h=\overline{h}$ a.e.\ in $\Yopen$.
\end{proof}

\begin{lemma}
  \label{lem:ae_to_Aessim}
  Suppose that $\Yopen \subseteq Y$ is open
  and $h\restrict_{\Yopen}$ is essentially bounded.
  Define $X := \essim h\restrict_{\Yopen}$.
  Suppose that $\Psi : \Yopen \times X \to \Rset^q$ is continuous and
  \begin{equation}
    \Psi\bigl(y, h(y)\bigr) = 0\ \text{for a.e. $y \in \Yopen$}.
    \label{eq:ae_identity}
  \end{equation}
  If $y_\star \in \Yopen$ and $u \in \Aessim(h; y_\star)$,
  then
  \begin{equation}
    \Psi(y_\star, u) = 0.
  \end{equation}
\end{lemma}

\begin{proof}
  There exists a set $\mN \subseteq \Yopen$ of measure zero such that
  \begin{equation}
    \text{$h(y) \in X$ and $\Psi\bigl(y, h(y)\bigr) = 0$
    for all $y \in \Yopen\,\setminus\,\mN$.}
    \label{eq:identity_off_N}
  \end{equation}
  For each integer $k \ge 1$, define
  \begin{equation}
    E_k := \{\, y \in \Yopen \cap B_{1/k}(y_\star) \,:\, \norm{h(y) -
    u} < 1/k \,\}.
  \end{equation}
  Since $u \in \Aessim(h; y_\star)$,
  the set $E_k$ has positive measure,
  so that $E_k\,\setminus\,\mN$ is nonempty.
  Choose $y_k \in E_k\,\setminus\,\mN$.
  Then $y_k \to y_\star$ and $h(y_k) \to u$ as $k \to \infty$.
  By Eq.~\eqref{eq:identity_off_N},
  \begin{equation}
    \Psi\bigl(y_k, h(y_k)\bigr) = 0
  \end{equation}
  for all $k \ge 1$.
  Taking $k \to \infty$ and using continuity of $\Psi$ yields
  $\Psi(y_\star, u)=0$.
\end{proof}

\bibliographystyle{amsplain}
\bibliography{biblio}

\end{document}